\documentclass[12pt,a4paper]{amsart}
\usepackage{graphicx,a4wide,array}
\usepackage{amsmath,amsthm,amsfonts,amssymb,color}
\usepackage[utf8]{inputenc}
\usepackage{dsfont}
\numberwithin{equation}{section}
\newtheorem{theorem}{Theorem}[section]
\newtheorem{lemma}[theorem]{Lemma}

\newtheorem{definition}[theorem]{Definition}

\usepackage{fancyhdr}
\newcommand\EGL{E_\mathrm{GL}}

\newcommand\C{\mathbb{C}}
\newcommand\loc{{\mathrm{loc}}}
\newcommand\bext{b_{\mathrm{ext}}}
\newcommand\Bext{B_{\mathrm{ext}}}
\newcommand{\gsim}{\gtrsim}
\newcommand{\lsim}{\lesssim}

\renewcommand\tilde{\widetilde}
\newcommand\calH{\mathcal{H}}
\def\Int{\mathrm{int}}

\def\nab{\nabla}
\newcommand\e{\varepsilon}

\newcommand\calD{\mathcal{D}}

\def\R{\mathbb{R}}
\def\Z{\mathbb{Z}}
\def\N{\mathbb{N}}
\def\C{\mathbb{C}}
\def\Div{\mathop{\rm div}}

\def\dist{{\rm dist}}

\newcommand{\uno}{\mathds{1}}

\begin{document}
  \title{Branched microstructures in the  Ginzburg-Landau  model of
    type-I superconductors}
    \author
{Sergio Conti, Felix Otto, and Sylvia Serfaty}
\maketitle
\begin{abstract} We consider the Ginzburg-Landau energy for a type-I superconductor in the shape of an infinite three-dimensional slab, with two-dimensional periodicity, with an applied magnetic field which is uniform and perpendicular to the slab. We determine
the optimal scaling law of the minimal energy in terms of the parameters of the problem,  when  the applied magnetic field is sufficiently small   and the sample sufficiently thick. 
This optimal scaling law is proven via ansatz-free lower bounds and an explicit branching construction
which refines further and further as one approaches the surface of the sample.
Two different regimes appear, with different scaling exponents.
In the first regime, the branching leads to an almost uniform magnetic field pattern on the boundary;
in the second one  the inhomogeneity  survives up to the boundary.
\end{abstract}
\section{Introduction}
Superconductivity, discovered in  1911 by Kamerlingh Onnes, is a phenomenon happening at low temperature in certain materials  which loose their resistivity and expel an applied magnetic field. The latter  is called the Meissner effect. More precisely, when the applied magnetic field is small, 
the sample is everywhere superconducting and completely expels the magnetic field, while when the magnetic field becomes larger, it partially penetrates the sample via regions of normal phase where the material is not superconducting. If the magnetic field is further increased, then superconductivity is completely destroyed and the sample behaves like a normal conductor.

The standard model for describing superconductivity is the Ginzburg-Landau functional, which  was introduced in the 1950's by Landau and Ginzburg on  a phenomenological basis. 
It was later justified based on microscopic quantum mechanical principles via the Bardeen-Cooper-Schrieffer (BCS) theory, which explains superconductivity through the appearance of ``Cooper pairs" of superconducting electrons. The Ginzburg-Landau model is a formal limit of the BCS model, and this derivation was accomplished rigorously in the recent work \cite{FrankSeiringer2012}.

The Ginzburg-Landau model  describes the state of the sample via a  complex-valued order parameter $u$.
The squared modulus of $u$
represents the local density of the Cooper pairs of  ``superconducting electrons".
In other words, $\rho:=|u|^2$ indicates whether one is in the normal phase $\rho\simeq 0$, or in the superconducting phase $\rho\simeq 1$.
The transition between $0$ and $1$  happens within relatively thin interfacial layers (or walls).
The order parameter $u$ is coupled with the  magnetic vector potential $A$, which yields  the magnetic field  $B:=\nab \times A $ induced in the sample. 
The Meissner effect can be roughly understood as the fact that the magnetic field $B$ can only exist in the normal phase $\rho=0$, or in other words 
\begin{equation}\label{apprxmeissner} \rho B\approx 0.\end{equation}
Another important property of superconductors 
is flux quantization. If we consider a closed circuit well inside the superconducting region 
(on a scale set by the penetration length $\lambda$), the contour integral of $A$ has to be an integer multiple of $2\pi$, in units of $\hbar/2e$.
This arises because $\nabla u$ is very close to $iAu$ and 
 $|u|$ is very close to 1.
Correspondingly the flux of the magnetic field $B$ through any surface with such boundary is quantized,
in the sense that it has to be an integer multiple of $2\pi$ (again, and for the rest of this paper, in units of $\hbar/2e$).

The Ginzburg-Landau functional in a three-dimensional region $Q_{L,T}:= (0,L)^2 \times (0,T)$  can be written, after 
appropriate non-dimensionalization,  as 
\begin{equation}\label{eqdefmodelintro}
  \EGL[u,A]:=\int_{Q_{L,T}} \left[ |\nabla_A u |^2
+ \frac{\kappa^2}{2} (1-|u|^2)^2\right] dx
+ \int_{Q_L\times \R} |\nabla\times A-\Bext|^2 dx.
\end{equation}
Here $\nab_A := \nab -i A $ denotes the covariant gradient,  $\Bext:=\bext e_3$ is the applied magnetic field, which is assumed to be uniform and vertical. The constant $\kappa>0$, usually called the Ginzburg-Landau parameter, is the ratio of the ``penetration length" (of the magnetic field in the sample) $\lambda$ and the ``coherence length" $\xi$. 
For a general presentation of superconductivity and the Ginzburg-Landau model, we refer to the standard physics textbooks, such as \cite{Tinkham1996,DeGennes,SST}.
 For further mathematical reference on the Ginzburg-Landau functional, one can see for example \cite{SandierSerfaty2007}.

Superconductors are usually classified in type-I and type-II superconductors, according to whether $\kappa<1/\sqrt{2}$ or $\kappa>1/\sqrt{2}$. In type-II superconductors, there is an intermediate regime,  for low applied magnetic fields, where the  penetration of the magnetic field happens along very thin vortex filaments, carrying an integer flux, and around which the sample is normal --- this is called the mixed phase. 
The size of the vortices is only limited by the flux quantization condition, and indeed in most situations each of them carries 
exactly one quantum of flux.
This is the regime studied in details in  \cite{SandierSerfaty2007} in dimension 2 and  in \cite{BJOS} in dimension 3.
  By contrast, in type-I superconductors the ratio of characteristic lengthscales $\kappa$  does not allow  these vortex-filaments to form and larger regions of normal phase appear, separated  interfaces (called walls) from the superconducting phase. 
  Each normal region in this case carries a magnetic flux much larger than the flux quantum.
  We will be interested only in the latter situation, and we will  assume that $\kappa$ is small enough, and also that 
  the applied field $\bext$ is much smaller than the critical field, which in the present units is $\kappa\sqrt2$. 

  The pattern of the normal phase arises from the competition of different effects. On the one hand,
  the interfacial energy favours a coarse structure in the interior of the sample. On the other hand,
  the magnetic energy outside the sample favours a fine-scale mixture close to the interface. Therefore,
  the optimal pattern is expected to branch, as predicted by Landau back in 1938 
   \cite{Landau38,Landau43}. This permits to combine a coarse pattern in the interior with an
   induced magnetic field $\nab \times A$ almost aligned with $\Bext $ at the surface, see Figure \ref{figbranching} for a sketch.
Experimentally, this is manifested   by complex patterns 
observed at the surface of the sample \cite{Prozorov2007,ProzorovGiannetta2005,ProzorovHobergCanfield2008,ProzorovHoberg2009}. 
This phenomenon of domain branching   occurs also in other areas of materials science, as for example ferromagnetic materials, where the magnetization pattern is constrained to oscillate between two opposite vectors
 \cite{Lifshitz44,Hubert67}, and in shape-memory alloys, 
where the strain can oscillate between finitely many values, corresponding to the different martensitic variants.
The average behavior of these branched patterns can be characterized via scaling laws:
one determines  how  the minimal energy
per cross-sectional area scales with the various parameters of the system, and shows that the optimal scaling of the 
energy can be achieved with branching-type patterns. This is usually rigorously established by showing  ansatz-free lower bounds and complementing them with the construction of  explicit branching patterns whose energy is estimated to have the same order in the parameters as the lower bound.
This was achieved  in martensites in 
 \cite{KohnMuller92,KohnMuller94,Conti00,CapellaOtto2009,CapellaOtto2012,Zwicknagl2014,ChanConti2015}
and in magnetic materials in \cite{CK98,CKO99,ViehmannOtto2010,Viehmanndiss}.

In the case of type-I superconductors, a similar program was carried out in 
\cite{ChoksiKohnOtto2004,ChoksiContiKohnOtto2008} for a simplified model: it is 
 a ``sharp-interface" version of the Ginzburg-Landau functional, where the order parameter $u$ is only represented via its modulus $\rho$, which in turn is only allowed to take  values in  $\{0,1\}$; at the same time the kinetic energy is replaced by a constant times the $BV$ norm of $\rho$, i.e., the perimeter of the set where $\rho=1$, see Section \ref{secsharpinterface} below for details. 
 This resulted in a full characterization of the phase diagram at the level of energy
scaling, and in particular led to the discovery of a new phase for very
small applied fields, see Figure \ref{figbranching} for a sketch.

\begin{figure}
 \begin{center}
  \includegraphics[height=5cm]{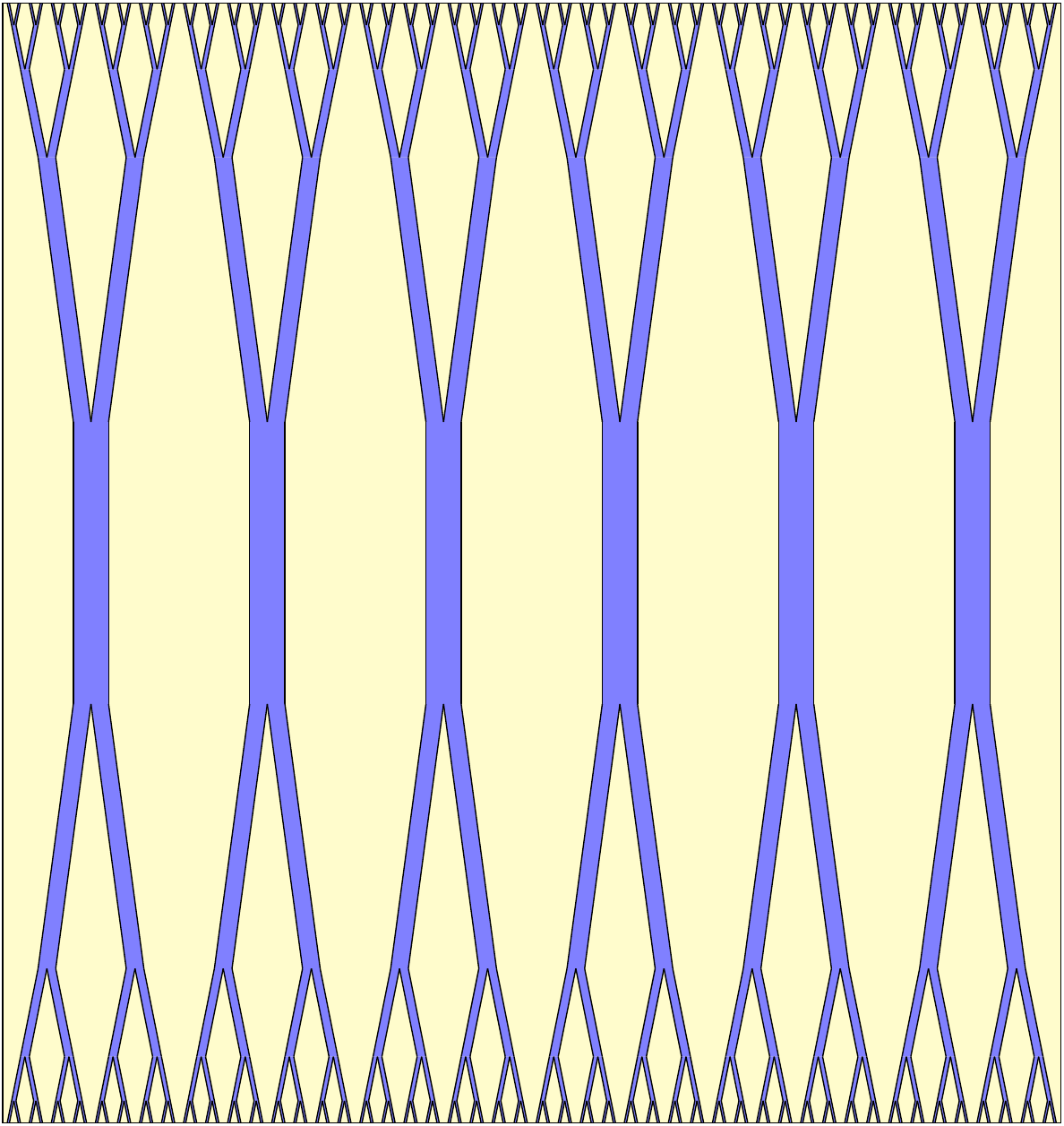}\hskip2cm
\includegraphics[height=5cm]{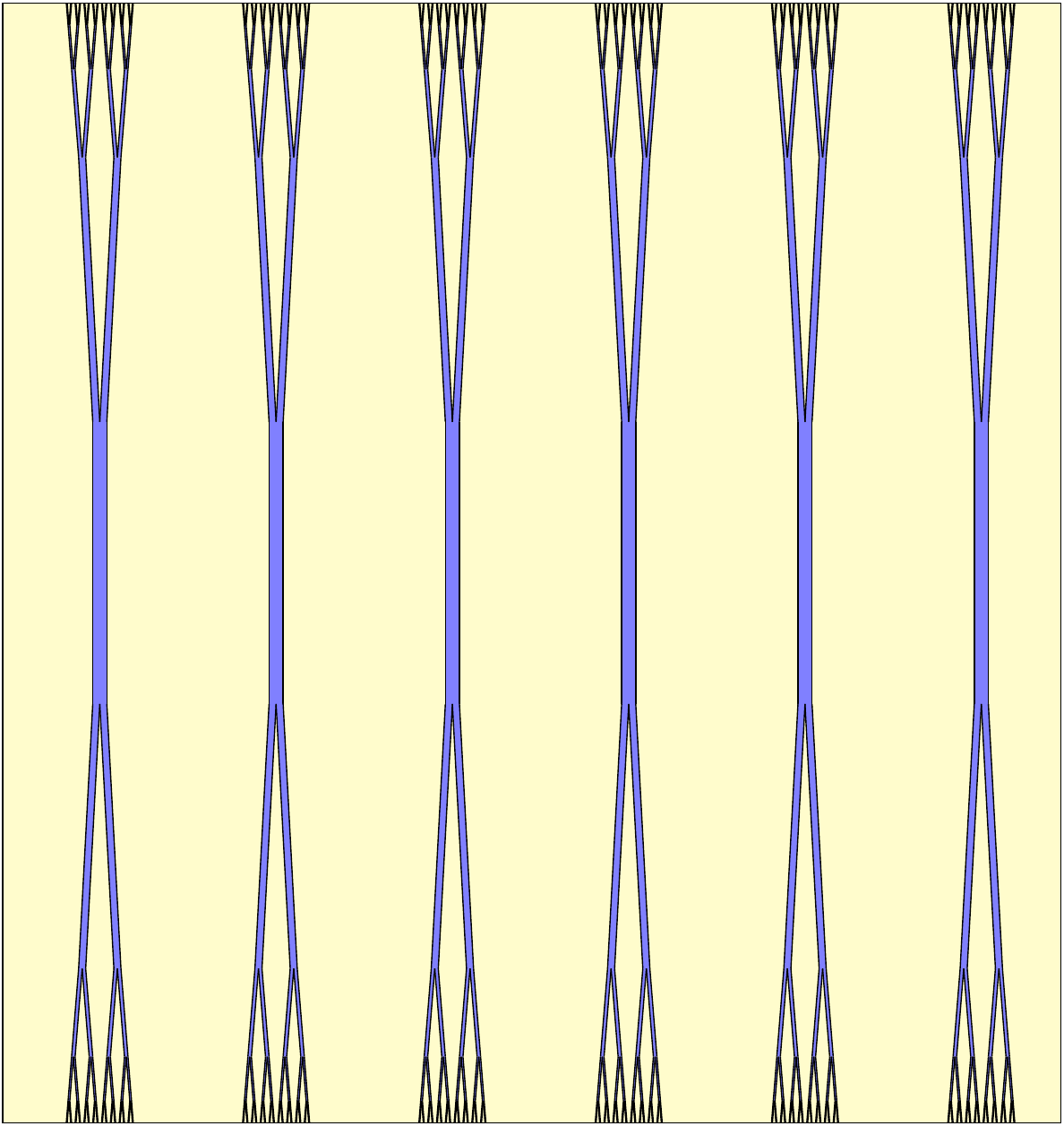}
\end{center}
\caption{Sketch of the flux patterns. Left: the regime with uniform flux on the surface, 
with energy proportional to $\bext^{2/3} \kappa^{2/3}  T^{1/3}L^2$, 
which is the optimal scaling if $\bext\ge \kappa^{5/7}/T^{2/7}$.
Right: the regime with flux concentration on the surface,
with energy proportional to 
$\bext \kappa^{3/7} T^{3/7} L^2$, which is the optimal scaling if 
$\bext\le \kappa^{5/7}/T^{2/7}$.}\label{figbranching}
\end{figure}

We study here the full Ginzburg-Landau model, as given in (\ref{eqdefmodelintro}), and determine
the scaling of the minimum energy per cross-sectional area in dependence of the problem parameters $\kappa$, $\bext$ and $T$.
We prove that the energy scaling is characterized by the same two regimes which had been found for the sharp-interface functional.
In fact, some of the ideas  of proof of \cite{ChoksiKohnOtto2004,ChoksiContiKohnOtto2008} carry over to the full model, once an appropriate splitting of the energy, involving a ``Bogomoln'yi operator" has been performed (cf. Section \ref{secprelim}).
The treatment of the lower bound contains several additional difficulties, mainly due to the fact that the Meissner condition
(\ref{apprxmeissner}) is only true ``on average'', in the sense that an appropriate weak norm is small. 
The constructions in the upper bound, at the same time, need to take into account the quantization condition, locally in each
tube, and to construct an order parameter with diffuse interfaces. 
The relationship with the simplified model is discussed in more detail  in Section \ref{secsharpinterface}.

We  work in an infinite periodic slab geometry, which is enough to understand the main surface branching features that we are interested in, hence the  choice of working in the domain  $Q_{L,T}:=Q_L\times(0,T)$ with $Q_L:=(0,L)^2$, with $L$ very large, and with horizontally periodic boundary conditions.
We recall that the Ginzburg-Landau functional is invariant under gauge-transformations:  two configurations $(u,A) $ and $(\hat u,\hat A)$  are called gauge-equivalent if there exists $\Phi \in H^2_{\mathrm{loc}}$ such that 
$$\left\{\begin{array}
{l}
u( x)= \hat u(x) e^{i \Phi(x)}\\
A(x) = \hat A(x)+ \nabla \Phi(x) \, .\end{array} \right.$$
The physical 
quantities are  gauge-invariant (i.e., invariant under a gauge-transformation). This includes the magnetic field $\nabla\times A$,  the energy, the density $\rho= |u|^2$ and the superconducting current defined as
\begin{equation}\label{defj}
j_A:=
\mathrm{Re} (-i\bar{u} \nabla_A u)=
\frac12 \left(- i\bar{u} \nab_A u + iu \overline{\nab_A {u}}\,\right)\,.
\end{equation} 

We will work in  the space 
$H^1_{\mathrm{per}}(Q_{L,T}\times \R;\C\times\R^3) $ defined  to be the set of $ (u,A)\in H^1_{\mathrm{loc}}$ such that  $\EGL[u,A]$ is finite and for every $\vec{k}\in \mathbb{Z}^2\times\{0\}$, $(u( \cdot + L \vec{k}), A (\cdot + L \vec{k}))$ is gauge-equivalent to $(u,A)$ (this periodic setting was rigorously formalized in \cite{Ode67}, see also \cite{Du99}).
All gauge-invariant quantities,
such as $\rho$, $j_A$ and $B$, are then   $Q_L$-periodic.
We stress that periodicity is only assumed in the first two variables. We will also call such pairs $(u,A)\in H^1_{\mathrm{per}}$ admissible.

Our main result, characterizing the energy in the regime of small applied fields $\bext$ and large and thick enough samples, is 
\begin{theorem}
For any   $\kappa,\bext,L,T>0$  such that
\begin{equation*}
  \bext L^2\in 2\pi\Z, \quad 8\bext\le  \kappa \le \frac12\,,\quad
  \kappa T\ge 1
\end{equation*}
and $L$ is sufficiently large (in the sense of \eqref{eqLadmissible}), one has
\begin{equation}\label{scalingresult}
  \min_{H^1_{\mathrm{per}}} \EGL[u,A]-(\kappa\sqrt2 \bext-\bext^2)L^2T\sim \min \left\{\bext \kappa^{3/7} T^{3/7} L^2, 
\bext^{2/3} \kappa^{2/3}  T^{1/3}L^2\right\} \,.
\end{equation}
\end{theorem}
The result will follow  from Theorem \ref{theolowerbGL} and Theorem \ref{theoupperbGL} below
using Lemma \ref{lemmaseparatebulk} to separate the bulk contribution.
The notation $a\sim b$ means that a universal constant $c>0$ exists, such that $c^{-1}a\le b\le ca$.

The scaling result \eqref{scalingresult} is in the end the same as in \cite{ChoksiKohnOtto2004,ChoksiContiKohnOtto2008}, after some rescaling of the lengths and magnetic field intensity. As in those works, the minimum in the right-hand side reflects the fact that two types of construction are needed, one corresponding to the regime where 
$\bext \kappa^{3/7} T^{3/7} L^2\ll \bext^{2/3} \kappa^{2/3}  T^{1/3}L^2$ and one corresponding to the opposite case.  In both cases, the construction that gives the optimal scaling law is that of a self-similar branching tree of normal region (where the magnetic field penetrates) which is symmetric with respect to the $x_3=T/2$ plane, and refines further and further as $x_3 $ approaches $0$ and $T$.
The optimal ``opening ratio" of the self-similar tree depends on the parameters of the problem and has to be chosen differently 
in the two regimes above.

A finer analysis in the asymptotic regime $\bext\to0$, including $\Gamma$-convergence
to a reduced model with energy concentrated on lines, will be discussed
in  \cite{ContiGoldmanOttoSerfaty}.

\medskip 

This paper is organized as follows. In Section \ref{secprelim} we show how the bulk contribution  energy can be algebraically separated, via the Bogomoln'yi operator,  and we define in 
\eqref{eqdefE} the functional $E$ on which we shall focus for most of the paper. In Section \ref{seclowerb}
we prove the lower bound, first for the sharp-interface version of the problem, and then for $E$. 
In Section \ref{secupperb} we prove the corresponding upper bounds; again we first work on
the sharp-interface problem and then extend the upper bound to the full  Ginzburg-Landau functional. 
\section{Preliminaries}\label{secprelim}
\subsection{Preliminary on notation}
We use a prime to indicate the first two components of a vector in $\R^3$, and
identify $\R^2$ with $\R^2\times\{0\}\subset\R^3$.
Precisely, for $a\in \R^3$ we write $a'=(a_1,a_2,0)\in\R^2\subset\R^3$;  given
two vectors $a,b\in\R^3$ we write  briefly $a'\times b'=(a\times b)_3 =
(a'\times b')_3$.  

We shall denote sections of $Q_{L,T}$ by
$Q(z)=Q_L\times\{z\}$, for integrals over $Q(z)$ we write $dx'$ instead of $d\calH^2(x')$.
In the entire paper we let $\kappa$, $\bext$, $L$, $T$ be positive parameters which obey
\begin{equation}\label{eqparameters}
  \bext\le \frac18 \kappa\,, \hskip1cm \kappa  \le \frac12\,,\hskip1cm\text{ and}\hskip1cm
  \kappa T\ge 1\,.
\end{equation}
In many parts we shall additionally require the quantization condition  $\bext L^2\in 2\pi\Z$.

By $a\lsim b$ or $b\gsim a$ we mean that a universal constant $c>0$ exists (which may change
from line to line but 
does not depend on the parameters of the problem) such that $a\le c b$. 
By $a\sim b$ we mean $a\lsim b$ and $b\lsim a$ (with two different implicit constants).

We denote by $H^{1/2}(Q_L)$ the space of traces of $Q_L$-periodic functions
$u\in H^1_\loc(Q_L\times(-\infty,0))$ with $\nabla u\in L^2$, and use the homogeneous norm
$\|u\|_{H^{1/2}(Q_L)}:=\inf 
\{\|\nabla u\|_{L^{2}(Q_L\times(-\infty,0))}\}$, where the infimum is taken over all possible extensions. We denote
by $H^{-1/2}(Q_L)$ its dual space.

\subsection{Separating the bulk energy}Our first step is to subtract the bulk contribution to $\EGL$, which will lead us to 
the definition of the energy $E$, which only contains the contributions of the microstructure.
The precise formula for $E$ is given in (\ref{eqdefE}) below. 
This  is done  as in \cite{ss} via an algebraic relation which involves the operator  $\calD_A$ defined as follows: 
\begin{equation}\label{DA} \calD_A^ku := (\nabla_A u)_{k+2} - i (\nabla_A u)_{k+1} \,,
\end{equation}
where components are understood cyclically (i.e., $a_k=a_{k+3}$). 
In particular, 
\begin{equation*}
  \calD_A^3u =   (\partial_2 u - i A_2 u) - i (\partial_1 u
 - i A_1 u) 
 = (\nabla_A u)_2 - i (\nabla_A u)_1 \,.
\end{equation*}
The operator $\calD_A$ corresponds to a ``creation operator" for a magnetic Laplacian in quantum mechanics. It was introduced in the context of Ginzburg-Landau by Bogomoln'yi  to prove the self-duality of the Ginzburg-Landau functional at $\kappa=1/\sqrt{2}$, cf. e.g \cite{jaffetaubes}. His proof relies on identities similar to the next one.

\begin{lemma}\label{lemmaformula}
With the notation above, one has
\begin{equation*}
  |\nabla'_A u|^2 =  |\calD_A^3u|^2 + \rho B_3 + \nabla'\times 
   j_A' \end{equation*}
and, for any $k=1,2,3$,
\begin{equation*}
   |(\nabla_A u)^{k+1}|^2 +    |(\nabla_A u)^{k+2}|^2= |\calD_A^ku|^2
   + \rho B_k + (\nabla\times j_A)_k  \,.
\end{equation*}
\end{lemma}
\begin{proof} We only prove the first relation, the other follows by
   relabeling coordinates. 
Notice that
\begin{equation*}
  \overline{\nabla_A u}=\nabla_{-A}\bar u\,.
\end{equation*}
We compute
 \begin{eqnarray*}
   |\calD_A^3u|^2 &=& 
\left|  (\nabla_A u)_2 - i (\nabla_A u)_1 \right|^2\\
 &=& 
\left(  (\nabla_A u)_2 - i (\nabla_A u)_1 \right)
\left( ( \overline{\nabla_{A} u})_2 + i (\overline{\nabla_{A}  u})_1
\right) \\ 
&=&|(\nabla_Au)_1|^2 + |(\nabla_Au)_2|^2 -
i(\nabla_Au)_1(\overline{\nabla_A u})_2 +
i(\nabla_Au)_2(\overline{\nabla_A u})_1\\
&=& |\nabla_A'u|^2 -i (\nabla'_Au)\times (\overline{\nabla_A' u})
\end{eqnarray*}
and 
\begin{eqnarray*}
  \nabla'\times (-i\bar u\nabla'_Au)&=&-i(\nabla' \bar u)\times \nabla_A' u
- i \bar u\nabla'\times(\nabla_A'u)
\\
&=& -i(\nabla' \bar u+iA'\bar u)\times(\nabla_A'  u)
-A'\bar u\times(\nabla_A'  u)  
-\bar u \nabla'\times (A'u)\\
&=&-i (\overline{\nabla_A'  u}) \times(\nabla_A'  u)
-A'\bar u \times(\nabla'u) -|u|^2\nabla'\times A' 
- (\nabla'u)\times (A'\bar u)\\
&=&-i (\overline{\nabla_A'  u}) \times(\nabla_A'  u)
-|u|^2\nabla'\times A'\,.
\end{eqnarray*}
Since the vector product is antisymmetric, the last expression is real.
Recalling the definition of $j_A$ from \eqref{defj}
we obtain
\begin{equation*}
  \nabla'\times j'_A
  =\mathrm {Re} \nabla'\times[ -i\bar u \nabla_A'u]
=-i (\overline{\nabla_A'  u}) \times(\nabla_A'  u)
-|u|^2\nabla'\times A'\,.
\end{equation*}
Adding terms concludes the proof.
\end{proof}

We now show that the flux of $B_3$ over every section is constant, due to 
the divergence-free condition.
\begin{lemma}\label{lemmaHmeno12}
    Let $B \in  L^2_\loc(\R^3 ;\R^3) $  be $Q_L$-periodic and obey 
    $\Div  B = 0 $. 
    Then the quantity 
    \begin{equation*}
	\varphi(z):=\int_{Q(z)}  B_3\,dx' 
    \end{equation*}
    does not depend on $z$. In particular, if $B - \bext e_3 \in L^2(Q_L\times \R;\R^3)$ (implied by the finiteness of $\EGL$) then for all $z\in \R$,
    \begin{equation*}
\int_{Q(z)}  B_3\,dx'  = \bext L^2.
    \end{equation*}
\end{lemma}
\begin{proof}
 By the periodicity condition we can test  the relation $\Div B=0$ (which is true since $B = \nabla \times A$)  with a function $\theta\in C^1_c(\R)$, depending on $x_3$ alone.  This yields that for any  $\theta \in C^1_c(\R)$
 \begin{equation*}
 0=  \int_{Q_L\times \R} B_3 \theta'(x_3) dx= \int_\R \varphi(z) \theta'(z) \, dz\,.
 \end{equation*}
 It follows that $\varphi$ is constant. 
 The second assertion follows from the fact that 
$\int_{Q(z)} B_3^2 \, dx' \ge \varphi(z)^2/L^2$.
\end{proof}

At this point we are ready to separate the bulk term, and define the microstructure functional we shall
study below.
\begin{lemma}\label{lemmaseparatebulk}
For every admisible pair $(u,A)$ and any parameter set obeying (\ref{eqparameters}) one has
\begin{equation*}
  \EGL[u,A]= (\kappa\sqrt2 \bext-\bext^2)L^2T + E[u,A]\,,
\end{equation*}
where
\begin{alignat}1
  E[u,A]&:=
\int_{Q_{L,T}} \left[(1-\kappa \sqrt2) |\nabla_A' u |^2 + \kappa\sqrt2
  |\calD_A^3u|^2 + 
   |\partial_3 u -
  i A_3 u|^2 \right]dx \nonumber \\
& +  \int_{Q_{L,T}}
  \left(B_3 - \frac{\kappa}{\sqrt2} (1-\rho)\right)^2dx  \nonumber \\
&+  \int_{Q_L\times \R} |B'|^2 dx +
   \int_{Q_L\times(\R\setminus(0,T))} |B_3 -\bext|^2dx \,,
   \label{eqdefE}
\end{alignat}
and as above $B=\nabla \times A$, $\rho=|u|^2$.
\end{lemma}
\begin{proof}
Lemma \ref{lemmaformula} implies
\begin{equation*}
|\nabla_A' u |^2 =   (1-\kappa \sqrt2) |\nabla_A' u |^2 + \kappa\sqrt2 
|\calD_A^3u|^2 
+\kappa\sqrt2\rho B_3 + \kappa\sqrt2\nabla'\times j'_A\,.
\end{equation*}
The last term integrates to zero by the
periodicity of $j_A$. By Lemma 
\ref{lemmaHmeno12}, the average of the normal component of the magnetic field $B_3$ 
is $\bext$. Therefore for 
each fixed $z$, we have
\begin{alignat*}1
  \int_{Q(z)} |\nabla_A' u |^2 dx' =   & \int_{Q(z)}\left[ (1-\kappa \sqrt2) |\nabla_A' u |^2 + \kappa\sqrt2 
|\calD_A^3u|^2  +\kappa\sqrt2(\rho-1) B_3 + \kappa\sqrt2 \bext\right] dx'\,.
\end{alignat*}
We substitute and obtain, using $\int_{Q_{L,T}} (B_3-\bext)^2dx=\int_{Q_{L,T}} (B_3^2-\bext^2)dx$, 
\begin{alignat*}1
  \EGL[u,A]&=\int_{Q_{L,T}} \left[(1-\kappa \sqrt2) |\nabla_A' u |^2 + \kappa\sqrt2
  |\calD_A^3u|^2 + \left(B_3 - \frac{\kappa}{\sqrt2} (1-\rho)\right)^2 \right] dx \\
& +\int_{Q_{L,T}} \left[|\partial_3 u -
  i A_3 u|^2 -\bext^2 +\kappa\sqrt2 \bext  \right]dx\\
&+ \int_{Q_L\times \R} |(\nabla\times A)'|^2 dx+
  \int_{Q_L\times(\R\setminus(0,T))} |(\nabla\times A)_3 -\bext|^2dx\,.
\end{alignat*}
Thus,  the bulk energy is $\kappa\sqrt2 \bext-\bext^2 $, and the result follows.
\end{proof}

\subsection{Construction of test functions}
One main ingredient  of the proof of the lower bound is the following
concentration lemma. This can be seen as a combination of truncation
and mollification, and is closely related to Lemma 3.1 from
\cite{ChoksiContiKohnOtto2008} and Lemma 2.1 from  \cite{ContiNiethammerOtto2006}.
We formulate this lemma in generic dimension, with $Q_L^n=(0,L)^n$, for $Q_L^n$-periodic 
functions. In
the following only the $n=2$ case is used.
For $f\in L^1_\loc(\R^n)$ and $r>0$, we define $f_r\in C^0(\R^n)$ by 
averaging over $r$-balls, 
  \begin{equation}\label{eqdefaverage}
    f_{r}(x):= \frac{1}{|B_{r}|}
    \int_{B_{r}(x)} f(y) dy\,.
  \end{equation}
Notice that this operation preserves periodicity.
\def\err{r}
\begin{lemma}\label{lemmatestfunction}
  Let $\chi\in L^1_\loc(\R^n)$, $Q_L^n$-periodic, $\chi\ge 0$, $0< \ell\le \err$. Then there
  is $\psi\in  L^1_\loc(\R^n)$, $Q_L^n$-periodic, such that
  \begin{enumerate}
  \item $\displaystyle \psi(x)\ge  \chi_\ell (x) $;
  \item $\displaystyle  \sup\psi\le \sup \chi$;
  \item $\displaystyle  \int_{Q_L^n}\psi\, dy\le \frac{2^n\err^n}{\ell^n} \int_{Q_L^n}\chi\, dy$;
  \item $\displaystyle  \err\sup|\nabla\psi|\lsim \sup\chi$;
  \item  $\displaystyle  \err\int_{Q_L^n}|\nabla \psi|\, dy\lsim
  \frac{\err^n}{\ell^n} \int_{Q_L^n} \chi\, dy$. 
  \end{enumerate}
\end{lemma}
For future reference we remark that these estimates immediately imply
\begin{equation}\label{eqlemmatestfunctl2}
 \|\psi\|_{L^2(Q^n_L)}^2\lsim \frac{r^n}{\ell^n} \|\chi\|_{L^\infty(Q^n_L)}\|\chi\|_{L^1(Q^n_L)}
 \text{ and }
 \|\nabla\psi\|_{L^2(Q^n_L)}^2\lsim \frac{r^{n-2}}{\ell^n} \|\chi\|_{L^\infty(Q^n_L)}\|\chi\|_{L^1(Q^n_L)}\,.
\end{equation}

\begin{proof}
 By homogeneity, it suffices to consider the case $\sup\chi=1$.
  Define 
  \begin{equation*}
    \psi(x):=\left(\min\left\{\frac{2^n\err^n}{\ell^n} \chi_{2\err}, 1\right\}\right)_\err(x) =
    \frac{1}{|B_\err|}     \int_{B_\err(x)} \min\left\{\frac{2^n\err^n}{\ell^n}\chi_{2\err}(y), 1\right\} dy\,,
  \end{equation*}
where we use the notation of (\ref{eqdefaverage}).
Clearly $\psi\le 1$, and (ii) follows. (iii) is immediate. 

To prove (i), observe first that $\chi_\ell\le \sup\chi= 1$.
Fix some $x$. Since $B_\ell(x)\subset B_{2\err}(z)$ for all $z\in
B_\err(x)$, we have
\begin{eqnarray*}
  \chi_\ell(x) &=& \frac{1}{|B_\ell|} \int_{B_\ell(x)} \chi \, dy\\
&\le &\frac{1}{|B_\ell|} \int_{B_{2\err}(z)} \chi  \, dy
= \frac{2^n\err^n}{\ell^n} \chi_{2\err}(z)\,.
\end{eqnarray*}
Therefore 
\begin{equation*}
  \chi_\ell(x)\le \min\left\{\frac{2^n\err^n}{\ell^n} \chi_{2\err}(z), 1\right\} \hskip1cm
  \forall   z\in B_\err(x)\,.
\end{equation*}
Averaging over $B_\err(x)$ we obtain
\begin{equation*}
  \chi_\ell(x)\le \frac{1}{|B_\err|} \int_{B_\err(x)}
 \min\left\{\frac{2^n\err^n}{\ell^n} \chi_{2\err}(z), 1\right\} dz=\psi(x)\,,
\end{equation*}
which proves (i).

Further, for any pair $x$, $z$ one has, writing
 $C_{\err}(x,z):=(B_r(x)\setminus B_r(x+z)) \cup ( B_r(x+z)\setminus B_r(x))$,
\begin{equation*}
  |\psi(x)-\psi(x+z)|\le \frac{|C_{\err}(x,z)|}{|B_{\err}|} 
\le  \frac{c_n}{\err}|z|\,,
\end{equation*}
which implies (iv). To prove (v), write analogously
\begin{equation*}
  |\psi(x)-\psi(x+z)|\le \frac{1}{|B_r|} \int_{C_{\err}(x,z)}
f\, dy
\end{equation*} 
where $f=\min\{2^nr^n\ell^{-n}\chi_{2r},1\}$.
Integrating in $x$ and estimating as above $|C_{\err}(x,z)|\lsim r^{n-1}|z|$ we get
\begin{equation*}
  \int_{Q_L^n}  |\psi(x)-\psi(x+z)|dx\le \int_{Q_L^n} f \frac{|C_{\err}(x,z)|}{|B_r|}dx
\lsim\frac{|z|}{r}  \int_{Q_L^n} f \, dx
\end{equation*}
and the proof is concluded.
\end{proof}
\def\no{
Note that the function constructed above is not smooth. 
A straightforward generalization is the following:
\begin{lemma}\label{lemmatestfunctionsmooth}
 Let $\chi\in L^1_\loc(\R^n)$, $Q_L^n$-periodic, $\chi\ge 0$, $0< \ell\le \err$. Then there
  is $\psi\in C^\infty(\R^n)$, $Q_L^n$-periodic, such that
  \begin{enumerate}
  \item $\displaystyle \psi(x)\gsim \chi_\ell(x)$;
  \item $\displaystyle  r^{|\alpha|} \sup |\nabla^\alpha\psi|\lsim  \sup 
  \chi$;
  \item $\displaystyle  r^{|\alpha|}
  \int_{Q_L^n}|\nabla^\alpha\psi|\, dx\lsim \frac{\err^n}{\ell^n}
  \int_{Q_L^n}\chi\, dx$.
  \end{enumerate}
The last two conditions hold for any multiindex $\alpha\in\N^n$, the constants may depend on $|\alpha|$.
\end{lemma}
\begin{proof}
  The proof is similar to the above, with the following difference: one takes a
  family of mollifiers $\phi_\e(x):=\e^{-n}\phi(x/\e)$, with
  $\phi(x)=\tilde\phi(|x|)\in C_c^\infty(B_1,[0,\infty))$, with $\tilde\phi>0$ on $[0,1/2]$.  The radius
  $2r$ should be replaced by $4r$, and all inequalities only survive
  up to a constant (depending on $\min_{[0,1/2]}\phi$). 
\end{proof}
}

\subsection{The sharp-interface functional}\label{secsharpinterface}
In closing this preliminary section, we introduce the sharp-interface version of the functional $E$ 
 and the corresponding function spaces. In the sharp-interface functional, 
a function denoted  $\chi$ represents the characteristic function of the normal phase, and is constrained to take values in $\{0,1\}$. Thus $\chi$ is formally the equivalent of $1-\rho$ and the approximate Meissner effect \eqref{apprxmeissner} is imposed via 
$$B(1-\chi)=0.$$

\begin{definition}\label{defadmisssharp}
  We say that a pair $B\in L^2_\loc(\R^3;\R^3)$, $\chi\in BV_\loc(\R^3;\{0,1\})$ is
  admissible for the sharp-interface functional if both of them are $Q_L$-periodic and
\begin{equation}\label{eqsidecondF}
\Div B=0 \text{ distributionally, and }B(1-\chi)=0 \text{ a.e.}
\end{equation}
The condition $\Div B=0$ is understood as $\int_{\R^3} B\cdot \nabla\theta\, dx=0$ 
for all test functions $\theta\in C^1_c(\R^3)$.

Given an admissible pair $(\chi,B)$ we set
\begin{equation}\label{eqdefF}
  F[\chi,B] := \int_{Q_{L,T}}  \kappa |D\chi|  +\int_{Q_{L,T}} \left[|B'|^2 + \chi\left(B_3
  - \frac{\kappa}{\sqrt2}\right)^2 \right]dx+
  \int_{Q_L\times[\R\setminus(0,T)]} |B-\bext e_3|^2 dx\,.
  \end{equation}
\end{definition}

This is the sharp-interface functional studied in 
\cite{ChoksiKohnOtto2004,ChoksiContiKohnOtto2008}.
In comparing with those papers,
it is important to notice that several quantities are scaled differently. 
In particular, lengths are rescaled by a factor of $L$, and magnetic fields  by
$\kappa/\sqrt2$. Precisely,
denoting by $\chi^*$, $B^*$, $E^*$, $\e^*$, $L^*$  and $b_a^*$ the objects used in 
\cite{ChoksiKohnOtto2004,ChoksiContiKohnOtto2008}, one has
\begin{equation*}
  \chi^*(x)=1-\chi(xL)\,,\hskip1cm B^*(x)=\frac{\sqrt2}{\kappa} B(xL)
\end{equation*}
and correspondingly
\begin{equation*}
E^*=\frac{2F}{L^3\kappa^2}\,,\hskip5mm
 \e^*=\frac2{\kappa L} \,,\hskip5mm
L^*=\frac{T}{L}\,,\hskip5mm 
b_a^*=\frac{\sqrt2}{\kappa}\bext\,.
  \end{equation*}

\section{Lower bound}\label{seclowerb}
\label{secoldfunctional}
\subsection{Lower bound with sharp interfaces}
In order to prepare some intermediate results and to explain the strategy of the proof in a simpler context, we
first prove the lower bound for the sharp-interface functional, recovering the result  from \cite{ChoksiContiKohnOtto2008}. 
In proving the lower bound the fields $\chi$ and $B$ will be fixed
 admissible functions,
 in the sense of Definition~\ref{defadmisssharp}, 
and we shall 
 simply denote by $F$ the total energy, and by $F(z)$ the part of the energy
 localized in the surface $Q(z)$, i.e., 
\begin{equation*}
  F(z):=\int_{Q(z)} \kappa |D'\chi|  +\int_{Q(z)} \left[  |B'|^2 + \chi\left(B_3
    - \frac{\kappa}{\sqrt2}\right)^2 \right]dx'\,.
\end{equation*}
Clearly $\int_0^T F(z) dz \le F[\chi,B]$.
In several lemmas we shall additionally focus on the case that the energy is bounded by
\begin{equation}\label{eqFzallgood}
F[\chi,B]\le \frac1{8} \kappa \bext L^2T\,,
\end{equation}
and that a $z\in (0,L)$ is given, so that
  \begin{equation}\label{eqFzzgood}
    F(z)\le \frac1{8} \kappa \bext L^2\,.
\end{equation}

The key strategy is to select a good section $Q(z)$ which has small energy. Since the energy is small,
and the flux of the magnetic field is the same on every section, the magnetic field necessarily
concentrates on a small subset, whose perimeter is controlled by the energy (interior estimate). 
At the same time, close to the surface the energy favours a uniformly distributed magnetic field
(exterior term). But ``moving around'' the  magnetic field as $z$ changes  is only possible, due to the divergence-free condition, if the tangential components $B'$ are nonzero, which are also penalized by the energy
(transport term). Making these three effects quantitative, and balancing them, leads 
to the lower bound.

\subsubsection{Equidistribution of the phases}
We first show  that the average value of $\chi$ across ``good'' sections
is the one that relaxation theory would predict, up to a factor. 
\begin{lemma}\label{lemmachiloc}
If the admissible pair $(\chi,B)$  obeys
(\ref{eqFzallgood}) and (\ref{eqFzzgood}) for some $z\in (0,T)$, then 
\begin{equation}\label{eqintchi1-2}
\int_{Q_{L,T}}  \chi \, dx\, \sim \,  \frac{\bext  L^2T}{\kappa}
\end{equation}
and
\begin{equation} \label{eqintchi-2}
\int_{Q(z)}  \chi \, dx'\, \sim \,  \frac{\bext  L^2}{\kappa}\,.
\end{equation}
\end{lemma}
 \begin{proof}
We start with the second assertion. Recalling that $\chi^2=\chi$, $\int_{Q(z)} (B_3-\bext)dx'=0$ (Lemma \ref{lemmaHmeno12}) and $B(1-\chi)=0$, we compute
   \begin{eqnarray*}
     \left|\int_{Q(z)}\left( \frac{\kappa}{\sqrt2} \chi-\bext\right)dx'\right|&=& 
     \left|\int_{Q(z)} \chi \left(\frac{\kappa}{\sqrt2}-B_3\right)dx'\right|\\
&\le&
     \left(\int_{Q(z)} \chi dx' \right)^{1/2}
\left(\int_{Q(z)}\chi\left(\frac{\kappa}{\sqrt2}-B_3\right)^2dx'\right)^{1/2}\\
&=&
     \left(\int_{Q(z)}\frac{\kappa}{\sqrt2} \chi dx' \right)^{1/2}
\left(\frac{\sqrt2}{\kappa}\int_{Q(z)}\chi\left(\frac{\kappa}{\sqrt2}-B_3\right)^2dx'\right)^{1/2}\\
&\le&
\frac14      \left|\int_{Q(z)} \frac{\kappa}{\sqrt2}\chi dx'\right| + \frac{\sqrt2}{\kappa}F(z)\,.
   \end{eqnarray*}
Therefore, recalling  (\ref{eqFzzgood}), 
\begin{equation*}
  \frac34 \int_{Q(z)}\frac{\kappa}{\sqrt2} \chi \, dx' \le \bext L^2 + \frac{\sqrt2}{\kappa}F(z)
  \le \frac54 \bext L^2
\end{equation*}
and
\begin{equation*}
  \frac54\int_{Q(z)} \frac{\kappa}{\sqrt2} \chi \, dx' \ge \bext L^2 - \frac{\sqrt2}{\kappa}F(z)
  \ge \frac34 \bext L^2\,.
\end{equation*}
   This concludes the proof of (\ref{eqintchi-2}). The proof of (\ref{eqintchi1-2}) is analogous, integrating over 
$Q_{L,T}$ instead of $Q(z)$.
 \end{proof}

\subsubsection{Interior term}
We show that on sections with small energy the magnetic field necessarily concentrates,
as captured by the test function $\psi$.
\begin{lemma}\label{lemmaoldinterior}
For any   $r\ge \ell>0$ and $z\in(0,T)$, if (\ref{eqFzzgood}) holds and  $\psi$ is the function 
constructed via Lemma \ref{lemmatestfunction} from $\chi(\cdot, z)$ then
\begin{equation*}
  \frac{\kappa}{\sqrt2}\int_{Q(z)} \chi\, dx' -
  \int_{Q(z)} B_3\psi\, dx'\lsim \ell F(z) + \left(
  \frac{r^2}{\ell^2} \frac{\bext L^2}{\kappa}\right)^{1/2} F^{1/2}(z)
\end{equation*}
for some universal $c>0$.
\end{lemma}
\begin{proof}
We write
\begin{equation*}
  B_3=B_3\chi = \frac{\kappa}{\sqrt2} \chi +
  \left(B_3-\frac{\kappa}{\sqrt2}\right)\chi\,. 
\end{equation*}
Testing with $\psi$ we get
\begin{eqnarray*}
  \int_{Q(z)} B_3\psi\, dx'&=& \frac{\kappa}{\sqrt2} \int_{Q(z)}\chi\psi\, dx' + 
 \int_{Q(z)}\left(B_3-\frac{\kappa}{\sqrt2}\right)\chi\psi\, dx'\,.
\end{eqnarray*}
Using Lemma \ref{lemmatestfunction}(i) we obtain
\begin{eqnarray*}
  \int_{Q(z)} \chi\psi \,dx'\ge \int_{Q(z)} \chi\chi_\ell \,dx'
&=&\int_{Q(z)}\chi^2\,dx' + \int_{Q(z)}\chi(\chi_\ell-\chi)dx'\\
&\ge& \int_{Q(z)} \chi \,dx'- \int_{Q(z)}
  |\chi-\chi_\ell| \,dx'
\end{eqnarray*}
and therefore
\begin{equation*}
 \frac\kappa{\sqrt2} \int_{Q(z)} \chi \,dx'
 \le \int_{Q(z)} B_3\psi\, dx'+ \frac\kappa{\sqrt2}\int_{Q(z)}
  |\chi-\chi_\ell| \,dx'-
  \int_{Q(z)}\left(B_3-\frac{\kappa}{\sqrt2}\right)\chi\psi\, dx'\,.
\end{equation*}
The second term can be estimated by
\begin{equation*}
  \int_{Q(z)}   |\chi-\chi_\ell| dx'\le \ell 
  \int_{Q(z)}   |D' \chi|\le \frac{\ell}{\kappa} F(z)\,,
\end{equation*}
and the last one by
\begin{eqnarray*}
\left|   \int_{Q(z)}\left(B_3-\frac{\kappa}{\sqrt2}\right)\chi\psi\,dx' \right|& \le &
\|\psi\|_{L^2(Q_L)}
\left\|\left(B_3-\frac{\kappa}{\sqrt2}\right)\chi\right\|_{L^2(Q(z))}\\
& \le&
\|\psi\|_{L^2(Q_L)} F^{1/2}(z)\,.
\end{eqnarray*}
Collecting terms we obtain
\begin{equation*}
   \frac{\kappa}{\sqrt2} \int_{Q(z)}\chi\, dx' 
 \le \int_{Q(z)} B_3\psi dx'+  \ell F(z)
  +\|\psi\|_{L^2(Q_L)} F^{1/2}(z)\,.
\end{equation*}
Using (\ref{eqlemmatestfunctl2}) and (\ref{eqintchi-2}), we also have
\begin{equation}\label{psil2}
  \|\psi\|_{L^2(Q_L)}\le \|\psi\|_{L^1(Q_L)}^{1/2}\lsim \left(
  \frac{r^2}{\ell^2} \frac{\bext L^2}{\kappa}\right)^{1/2}\,.
\end{equation}
This concludes the proof.
 \end{proof}

\subsubsection{Transport term}\label{secoldtransport}
We now relate the value of $B_3$ over different sections by exploiting the $\int |B'|^2dx$ term in the energy. Since we have $\Div B=0$ we may write 
$$\partial_{3} B_3 + {\Div}' B'=0$$
i.e. $B'$ can be seen as the flux transporting $B_3$.
Since $B_3$ takes, approximately, only the two values 0 and $\kappa/\sqrt2$, up to a
factor we can understand $B'$ as the velocity with which $B_3$ is transported.
Thus  we call $\int |B'|^2dx$ the transport term by analogy with the Benamou-Brenier 
formula for the Wasserstein distance in optimal transport.
\begin{lemma}\label{LemmaMongeKanto}
    Let  $B \in  L^2_\loc(\R^2\times(0,T) ;\R^3) $, $Q_L$-periodic,  such that  $\Div
    B=0$. Then for any $z_1, z_2\in [0,T]$ one has
\begin{equation}\label{eqbprime}
\left|\int_{Q_L} \left(B_3(\cdot, z_1)\, -\, B_3(\cdot,z_2)\right)
 \,  \psi\, dx'\right|
\le  \|\nabla\psi\|_{L^\infty}\, 
 \int_{Q_L \times (z_1,z_2)}  |B'| \, dx
\end{equation}
for any  $\psi\in W^{1,\infty}(\R^2)$, $Q_L$-periodic.  The values of $B_3$ on the sections
 are understood as traces, which exist since $\Div B=0$.
\end{lemma}
\begin{proof} This is the same as \cite[Lemma 2.2]{ChoksiContiKohnOtto2008},
for completeness we give here the short argument.
 We can assume without loss of generality that $z_1<z_2$.
  We compute, using $\Div B=0$ and  the $Q_L$-periodicity of $B$ and
  $\psi$, 
  \begin{alignat*}1
     \int_{Q_L} [B_3(z_2)-B_3(z_1)]\psi \,dx'
=&\int_{Q_L\times(z_1,z_2)}
\frac{\partial B_3}{\partial x_3} \psi\,dx
=-\int_{Q_L\times(z_1,z_2)}
\nabla'\cdot B' \psi\,dx\\ 
=&\int_{Q_L\times(z_1,z_2)}
 B'\cdot\nabla\psi\,dx
\le \|\nabla\psi\|_{L^\infty(Q_L)}\int_{Q_{L}\times(z_1,z_2)} |B'|\,dx\,.
  \end{alignat*}
This, together with the same estimate with $-\psi$ in place of $\psi$, concludes the proof.
\end{proof}
Since $B'$ is nonzero only in a small part of the volume, the embedding of $L^1$ into $L^2$ gives
an additional factor proportional to $\bext$.
\begin{lemma}\label{lemmatransportsharp}
Let $(\chi,B)$ be an admissible pair which fulfills (\ref{eqFzallgood}), 
$z\in (0,T)$ such that (\ref{eqFzzgood}) holds. Then the function
$\psi$ from  Lemma 
\ref{lemmaoldinterior} fulfills
  \begin{equation*}
 \left| \int_{Q(z)} B_3\psi \, dx'-
  \int_{Q(0)} B_3\psi\, dx'\right|\lsim\frac1r
  \left( \frac{\bext L^2T}{\kappa}\right)^{1/2}
F^{1/2}\,.    
  \end{equation*}
\end{lemma}
\begin{proof}
 By Lemma \ref{LemmaMongeKanto} we have
 \begin{eqnarray*}
    \left| \int_{Q(z)} B_3\psi\,dx' -
  \int_{Q(0)} B_3\psi\,dx'\right| &\le&
\|\nabla\psi\|_{L^\infty} \int_{Q_{L,T}} |B'|\,dx\,.
 \end{eqnarray*}
We estimate
\begin{eqnarray*}\int_{Q_{L,T}}  |B'|\,dx=
\int_{Q_{L,T}} \chi |B'|\,dx
&\le&   \left( \int_{Q_{L,T}} \chi\,dx\right)^{1/2}
\left( \int_{Q_{L,T}} |B'|^2\,dx\right)^{1/2}\\
&\lsim&  \left( \frac{\bext L^2T}{\kappa}\right)^{1/2}
F^{1/2}\,.
 \end{eqnarray*}
To conclude the proof it suffices to recall that $\|\nabla\psi\|_{L^\infty}\lsim1/r$.
\end{proof}

\subsubsection{Exterior term}
Finally, we show that the energy outside the sample penalizes configurations
with a magnetic field that oscillates strongly at the boundary, as measured by the $H^{-1/2}$ norm.
\begin{lemma}\label{lemmaextsharp}
For all admissible $(\chi,B)$, and any $\psi\in H^{1/2}_\loc(\R^2)$, $Q_L$-periodic, 
it holds that
  \begin{equation*}
   \left| \int_{Q(0)} (B_3-\bext)\psi \,dx' \right|\lsim   \|\psi\|_{H^{1/2}(Q_L)}F^{1/2}   \,.
  \end{equation*}
If (\ref{eqFzzgood}) holds for some  $z\in(0,T)$ and
 $r\ge \ell>0$ and  $\psi$ are as in Lemma \ref{lemmaoldinterior}, then
\begin{equation*}  
  \left|  \int_{Q(0)} (B_3-\bext)\psi \,dx'\right|\lsim  \left(\frac{r \bext L^2}{\ell^2\kappa}\right)^{1/2}F^{1/2}  \,.
  \end{equation*} 
\end{lemma}
\begin{proof}
Let $\hat\psi\in H^1_\loc(\R^3)$ be $Q_L$ periodic, such that
$\hat\psi(\cdot, 0)=\psi$ in the sense of traces and
$\|\psi\|_{H^{1/2}(Q_L)}=\|\nabla \hat\psi\|_{L^2(Q_L\times(-\infty,0))}$
(we recall that we are using the homogeneous $H^{1/2}$ norm).
Since $\Div B=0$ on $(-\infty,0)\times Q_L$, using periodicity we obtain
  \begin{equation*}
  \int_{Q_L\times(-\infty,0)} (B-\Bext)\cdot \nabla\hat\psi\, dx = 
  \int_{Q(0)} (B-\Bext) \hat\psi \cdot e_3 \,dx'\,.
  \end{equation*}
  Therefore
  \begin{equation*}
  \left|  \int_{Q(0)} (B_3-\bext)\psi\, dx'\right|\le 
    \|B-\Bext\|_{L^2(Q_L\times(-\infty,0))}
    \|\nabla \hat\psi\|_{L^2(Q_L\times(-\infty,0))}\le
    F^{1/2} \|\psi\|_{H^{1/2}(Q_L)}\,.
  \end{equation*}
  This proves the first assertion. It remains to estimate $\|\psi\|_{H^{1/2}(Q_L)}$.
  By interpolation and (\ref{eqlemmatestfunctl2}) we have
  \begin{alignat*}1
    \|\psi\|_{H^{1/2}(Q_L)}&\lsim
    \|\psi\|_{L^2(Q_L)}^{1/2}
    \|\nabla\psi\|_{L^2(Q_L)}^{1/2}
    \lsim
    \left(\frac{r^2}{\ell^4}
    \|\chi\|_{L^1(Q_L)}^2\right)^{1/4}\,.
     \end{alignat*}
     Recalling  \eqref{psil2}  we conclude
  \begin{alignat*}1
    \|\psi\|_{H^{1/2}(Q_L)}&\lsim
\left(\frac{r \bext L^2}{\ell^2\kappa}\right)^{1/2}\,.
  \end{alignat*}
  This concludes the proof.
\end{proof}

\subsubsection{Derivation of the lower bound}
From what precedes, we deduce the lower bound result for the sharp-interface functional.
\begin{theorem}\label{theochilower1}
For any admissible pair $(\chi,B)$ (in the sense of Definition \ref{defadmisssharp}) and for all
$\bext,\kappa,L,T>0$ such that
\begin{equation*}
  8\bext\le \kappa \le \frac12\hskip1cm\text{ and }\hskip1cm
  \kappa T\ge 1\,,
\end{equation*}
one has
\begin{equation*}
 F[\chi,B]\gsim \min \left\{\bext \kappa^{3/7} T^{3/7} L^2, 
\bext^{2/3} \kappa^{2/3}  T^{1/3}L^2\right\} \,.
\end{equation*}
\end{theorem}
\begin{proof}
Let $F=F[\chi,B]$.
If $F\ge \bext\kappa  TL^2/8$ then, since $\kappa T\ge 1$,
we obtain  $F\ge  \bext(\kappa  T)^{3/7}L^2/8$ and the proof is concluded. We can therefore
assume that  (\ref{eqFzallgood}) holds.
By a mean-value argument, we may choose $z\in(0,T)$ such that
$F(z)\le F/T$, so that (\ref{eqFzzgood}) holds as well.
We start by constructing $\psi$ as in Lemma \ref{lemmaoldinterior}, namely,
with Lemma \ref{lemmatestfunction} applied to $\chi(\cdot, z)$ for some  $0< \ell\le r$ chosen below.
One key observation is that, since $\psi$ depends only on $x'$,
\begin{align*} \int_{Q(z)} B_3\psi \,dx'= &\int_{Q_L}\bext\psi  \,dx'+ \left[\int_{Q(z)}B_3\psi\,dx'
  -\int_{Q(0)}B_3\psi\,dx'\right]
  + \int_{Q(0)} (B_3-\bext)\psi\,dx'\,.
\end{align*}
With Lemma \ref{lemmatransportsharp}  and Lemma \ref{lemmaextsharp} this can be tranformed into
\begin{align}\nonumber \int_{Q(z)} B_3\psi \,dx'\le  &\int_{Q_L}\bext\psi  \,dx'
 + c\frac1r \left( \frac{\bext L^2T}{\kappa}\right)^{1/2} F^{1/2}
+  c  \left(\frac{r \bext L^2}{\ell^2\kappa}\right)^{1/2} F^{1/2}    
 \,,
\end{align}
Lemma \ref{lemmatestfunction}(iii) gives
\begin{equation*}
    \int_{Q_L}\bext\psi \,dx'\le \bext \|\psi\|_{L^1(Q_L)}\le 4\frac{r^2}{\ell^2} \bext \int_{Q(z)} \chi \, dx'\,.
\end{equation*}
Recalling Lemma \ref{lemmaoldinterior},
\begin{equation*}
    \frac{\kappa}{\sqrt2}\int_{Q(z)} \chi\, dx' 
    \le \int_{Q(z)} B_3\psi\, dx'+
  c \ell F(z) + c \left(
  \frac{r^2}{\ell^2} \frac{\bext L^2}{\kappa}\right)^{1/2} F^{1/2}(z)\,.
\end{equation*}
Combining the last three estimates gives
\begin{align*}
 \frac{\kappa}{\sqrt2}\int_{Q(z)} \chi\, dx'\le &
 4\frac{r^2}{\ell^2} \bext \int_{Q(z)} \chi \, dx'
 + c\frac1r \left( \frac{\bext L^2T}{\kappa}\right)^{1/2} F^{1/2}
+c  \left(\frac{r \bext L^2}{\ell^2\kappa}\right)^{1/2} F^{1/2}   
 \\
&  + c \ell F(z) +c \left(
  \frac{r^2}{\ell^2} \frac{\bext L^2}{\kappa}\right)^{1/2} F^{1/2}(z)\,.
\end{align*}
Assume now that 
\begin{equation}\label{eqdefadmissrl}
  0<\ell\le r\le \left(\frac{\kappa}{8\bext}\right)^{1/2} \ell\,,
\end{equation}
so that the coefficient of $\int_{Q(z)} \chi \, dx'$ in the right-hand side is smaller than the one on the left-hand side.
This is possible, since we assumed $\bext\le \kappa/8$.
Then, recalling (\ref{eqintchi-2}) and $F(z)\le F/T$,
\begin{align}\nonumber 
\bext L^2\lsim& 
 \frac1r  \left( \frac{\bext L^2T}{\kappa}\right)^{1/2} F^{1/2} 
+\left(\frac{r \bext L^2}{\ell^2\kappa}\right)^{1/2} F^{1/2} 
+  \ell \frac FT 
+ \left(  \frac{r^2}{\ell^2} \frac{\bext L^2}{\kappa}\right)^{1/2} \frac{F^{1/2}}{T^{1/2}}\,.
\end{align}
At least one of the terms in the right-hand side has to be at least one-quarter of the 
one on the left, and therefore for all pairs $(r,\ell)$
which obey 
(\ref{eqdefadmissrl}), we have
\begin{equation*}
  F\gsim \min\left\{
  \bext L^2 \frac{r^2 \kappa}{T},
  \bext L^2 \frac{\ell^2\kappa}{r}, 
  \bext L^2\frac{T}{\ell},
  \bext L^2\frac{\ell^2\kappa T}{r^2}
  \right\}.
\end{equation*}
Equivalently, 
\begin{equation}\label{eqfinalboundF}
  F\gsim \kappa \bext L^2T \min\left\{
  \frac{r^2 }{T^2},
  \frac{\ell^2}{rT},
  \frac{1}{\kappa\ell}, 
  \frac{\ell^2}{r^2} 
  \right\}.
\end{equation}
We finally have to choose $r$ and $\ell$, and check that in each case
some terms give the optimal bound, and the others are
irrelevant. 
Balancing the first three terms we obtain 
\begin{equation*}
  \ell=T^{4/7}\kappa^{-3/7}\,, \hskip5mm r=T^{5/7}\kappa^{-2/7}\,.
\end{equation*}
This choice is admissible only if  (\ref{eqdefadmissrl}) is satisfied, which since $\kappa T\ge1$ is equivalent to
 $(\kappa T)^{1/7}\le  (\kappa/8\bext)^{1/2}$. In this case,
(\ref{eqfinalboundF}) becomes
\begin{equation*}
  F\gsim
 \kappa^{3/7}\bext L^2T^{3/7} 
 \min\left\{
 1,1,1, (\kappa T)^{2/7}\right\}
\end{equation*}
and since $\kappa T\ge 1$ the assertion holds.

If instead $(\kappa T)^{1/7}\ge  (\kappa/8\bext)^{1/2}$, we choose
$r=\ell (\kappa/8\bext)^{1/2}$. Inserting this 
into (\ref{eqfinalboundF}) and then balancing the first and third term yields, after some rearrangement,
\begin{eqnarray*}
  F&\gsim&
 \kappa^{2/3}\bext^{2/3}L^2T^{1/3} \min\left\{
1,
(\kappa T)^{1/3} \left(\frac{\bext}{\kappa}\right)^{7/6},
1,
(\kappa T)^{2/3} \left(\frac{\bext}{\kappa}\right)^{4/3}
\right\}\,.
\end{eqnarray*}
We observe that since $\kappa T\ge1$ and 
 $(\kappa T)^{1/7}\ge(\kappa/8\bext)^{1/2}$ all terms of the form
 \begin{equation*}
   (\kappa T)^{\alpha} \left(\frac{\bext}{\kappa}\right)^{\beta}
 \end{equation*}
with $\alpha\ge\frac27\beta$ are bounded from below. This concludes
the proof.
  \end{proof}

\subsection{Lower bound for the Ginzburg-Landau functional}\label{lowerboundfull}
The proof is structured in a similar way as the one for the sharp-interface functional, but contains several
additional difficulties. In particular, when working with diffuse interfaces we can enforce the Meissner condition only in a weak sense, see 
Section \ref{secmeissneravg}. This generates difficulties both in the interior estimate and in the
transport term. Additionally, the energy does not directly control the size of the boundary
of the normal phase, and a suitable estimate needs to be formulated and proven (Section 
\ref{secfullsurface}). The only term which can be treated in the same way is the one
corresponding to the external field.

We consider a pair $(u,A)\in H^1_{\mathrm{per}}$ with $\EGL[u,A]$ finite, hence $E[u,A]$ finite. 
By Lemma \ref{lemmarhominoredi1} below, we can assume that $\rho\le 1$ pointwise.
For any $z\in(0,T)$, we  denote by
$E(z)$ the energy contained in the section
 $Q(z):=Q_L\times\{z\}$, 
\begin{alignat*}1
  E(z):=&
\int_{Q(z)}\left[ (1-\kappa \sqrt2) |\nabla_A' u |^2 + \kappa\sqrt2
  |\calD_A^3u|^2 + \left(B_3 - \frac{\kappa}{\sqrt2} (1-\rho)\right)^2 \right] dx'\\
& +  \int_{Q(z)}\left[ |B'|^2+|\partial_3 u -
  i A_3 u|^2  \right] dx'\,.
\end{alignat*}
We write for brevity $E=E[u,A]$. 
We recall that $\rho=|u|^2$ and define
\begin{equation}\label{eqdefchi}
  \chi := (1-\rho)^2\,.
\end{equation}

\subsubsection{Normalization of the density}
We first show that we can assume without loss of generality that $|u|\le 1$, or, equivalently, $\rho\le 1$.
\begin{lemma}\label{lemmarhominoredi1}
  Let  $(u,A)$ be an admissible pair. Then,
\begin{equation*}
 |\nabla(\rho^{1/2})|\le|\nabla_Au|\,,\hskip5mm 
|\calD_A'u|\le\sqrt2|\nabla_Au|\,.
  \end{equation*}
Let $\tilde u:\R^2\times(0,T)\to\C$ be defined by
\begin{equation}
  \tilde u(x) := 
  \begin{cases}
    u(x) & \text{ if }|u|(x)\le 1\\
    \displaystyle\frac{u}{|u|}(x) & \text{ if }|u|(x)> 1\,.
  \end{cases}
\end{equation}
Then $(\tilde u, A)$ is in $H^1_{\mathrm{per}}$ and $E[\tilde u,A]\le E[u,A]$.
\end{lemma}
\begin{proof}By density it suffices to prove all assertions under the additional assumption that $u\in C^1$, at points where $\rho\ne0$.
Writing locally  $u=\rho^{1/2} e^{i\theta}$ one obtains
\begin{equation*}
    \nabla_Au=e^{i\theta}\left[ \nabla\rho^{1/2} + i \rho^{1/2} (\nabla\theta -A)\right]\,,
  \end{equation*}
hence locally
\begin{equation}\label{naba}
  |\nabla_Au|^2=|\nabla\rho^{1/2}|^2+\rho|\nabla\theta-A|^2\,.
\end{equation} 
The second assertion follows immediately from the definition of $\calD_A$ in (\ref{DA}).
It also  follows directly from \eqref{naba} that  $|\nabla_A \tilde u|\le |\nabla_A u|$, and hence $\EGL[\tilde u,A]\le \EGL[u,A]$. Since $E$ and $\EGL$
differ by a constant, we conclude  that $E[\tilde u,A]\le E[u,A]$.
\end{proof}
\subsubsection{Meissner effect ``on average''}\label{secmeissneravg}
In the reduced model we had the compatibility condition
$B(1-\chi)=0$. In the true model, we  expect as in \eqref{apprxmeissner} that $B\rho$ is, in some
sense, small. The next lemmas make this quantitative, in appropriate
weak norms. 
\begin{lemma}\label{lemmarhoB3}
Let $(u,A)$ be admissible, and assume (\ref{eqparameters}) and $\rho\le 1$.
For any $z$, we have
  \begin{equation}\label{eqsectionrhohunemnorhodue}
    \left|\int_{Q(z)} \rho B_3 \,dx'\right| \le 8E(z)
  \end{equation}
and
  \begin{equation}\label{eqsectionrhohunemnorho}
    \left|\int_{Q(z)} \rho B_3 (1-\rho)\,dx'\right| \le 16\, E(z)\,.
  \end{equation}

Further, for any $\varphi\in L^\infty\cap W^{1,2}_\loc(\R^2\times(0,T))$, $Q_L$-periodic and such that
$\varphi=0$ on $Q(0)$ and $Q(T)$ (in the sense of traces), and any $k=1,2,3$, we
have 
\begin{equation}\label{eqhcontrollatutti}
 \left| \int_{Q_{L,T}} \rho B_k\varphi \,dx\right| \lsim E \|\varphi\|_{L^\infty(Q_{L,T})} +
  E^{1/2} \|\nabla\varphi\|_{L^2(Q_{L,T})}\,.
\end{equation}
\end{lemma}
\begin{proof}
We start with (\ref{eqsectionrhohunemnorhodue}).  By Lemma \ref{lemmaformula},
  \begin{equation*}
    \int_{Q(z)} \rho B_3 \,dx'= 
    \int_{Q(z)} \left[|\nabla_A'u|^2-\nabla'\times
    j'_A-|\calD_A^3u|^2\right]dx'\,. 
  \end{equation*}
  The integral of $\nabla'\times j'_A$ is zero, since $j'_A$ is
  $Q_L$-periodic and only in-plane derivatives appear. The other terms
  are bounded by the energy, i.e.,
  \begin{equation*}
   \frac12 \int_{Q(z)} |\calD_A^3u|^2\,dx'\le
    \int_{Q(z)} |\nabla_A'u|^2\,dx'\le \frac{E(z)}{1-\kappa\sqrt2}\le 4E(z)\,,
  \end{equation*}
  where we used $\kappa\le 1/2$.
 This concludes the proof of 
 (\ref{eqsectionrhohunemnorhodue}).

The argument for (\ref{eqsectionrhohunemnorho}) is similar. We write
  \begin{equation*}
    \int_{Q(z)} \rho B_3 (1-\rho)\,dx'= 
    \int_{Q(z)} \left[|\nabla_A'u|^2-\nabla'\times
    j'_A-|\calD_A^3u|^2\right](1-\rho)\,dx'\,. 
  \end{equation*}
The terms with $\nabla_Au$ and $\calD_Au$  can be estimated as above. The term
  with $\nabla'\times j_A'$
   however needs more care. Since $j_A$ and $\rho$ are  $Q_L$-periodic, 
 $|j_A|\le \rho^{1/2}|\nabla_A  u|$, $\rho \le 1$ and $|\nabla\rho^{1/2}|\le|\nabla_Au|$,  an
 integration by parts leads to
  \begin{equation*}
    \left|\int_{Q(z)} (1-\rho) \nabla'\times j_A'\,dx'\right|=
    \left|2\int_{Q(z)} \rho^{1/2} \nabla'\rho^{1/2} \times  j_A'\,dx'\right|
    \le \frac{2}{1-\kappa \sqrt2}E(z)\,.
  \end{equation*}
This proves (\ref{eqsectionrhohunemnorho}).

Finally, we consider  (\ref{eqhcontrollatutti}). Here we include the
other components, and 
we required that the localization function $\varphi$ vanishes on the top and bottom
boundaries (where we have no periodicity). We compute
\begin{eqnarray*}
   \left| \int_{Q_{L,T}} \rho B_k\varphi \,dx\right| &=&
\left|    \int_{Q_{L,T}} \left[
|\nabla_A^{(k+1)} u|^2+|\nabla_A^{(k+2)} u|^2 - |\calD_A^{(k)}u|^2 - (\nabla\times j_A)_k
\right]\varphi\,dx\right| \\
 &\le&
 16 E \|\varphi\|_{L^\infty} + \int_{Q_{L,T}} |j_A|\,|\nabla\varphi|\,dx\\
 &\le& 16E \|\varphi\|_{L^\infty}+
 \frac{1}{(1-\kappa\sqrt2)^{1/2}}
  E^{1/2} \|\nabla\varphi\|_{L^2(Q_{L,T})}\,.
\end{eqnarray*}
\end{proof}

\subsubsection{Surface energy}\label{secfullsurface}
We now show how the surface energy can be recovered from the functional. This arises  from the combination of a $|\nabla \rho^{1/2}|^2$ and a $(1-\rho)^2$ term, but the latter is not directly present in the energy, and needs first to be reconstructed.
\begin{lemma}
Let $(u,A)$ be admissible, and assume (\ref{eqparameters}) and  $\rho\le 1$.
  For every $z\in(0,T)$ we have
  \begin{equation}\label{eqrho6nc}
    \kappa^2\int_{Q(z)} \rho(1-\rho)^2 \,dx'\lsim E(z)\,.
  \end{equation}
The function $\chi$ defined in (\ref{eqdefchi}) satisfies
  \begin{equation}\label{eqnablaprimechi}
    \kappa\int_{Q(z)} |\nabla'\chi|\,dx' \lsim E(z)
  \end{equation}
and, for any $\ell>0$,
\begin{equation}\label{eqchichiell}
  \int_{Q(z)} |\chi_\ell-\chi|\,dx'
  \lsim\frac{\ell}{\kappa}  E(z)\,.
\end{equation}
Here $\chi_\ell$ is the average of $\chi$ over $\ell$-balls, defined in (\ref{eqdefaverage}).
\end{lemma}
\begin{proof}
We have
\begin{eqnarray*}
  \left(B_3 - \frac{\kappa}{\sqrt2} (1-\rho)\right)^2 
&\ge& \rho \left(B_3 - \frac{\kappa}{\sqrt2} (1-\rho)\right)^2\\
&\ge& \frac{\kappa^2}{2} \rho(1-\rho)^2 - \kappa \sqrt2\rho B_3(1-\rho)\,.
\end{eqnarray*}
Therefore, recalling (\ref{eqsectionrhohunemnorho}),  we have 
\begin{eqnarray*}
\frac{\kappa^2}{2}  \int_{Q(z)} \rho(1-\rho)^2 \,dx'\le E(z)+16\sqrt2 \kappa E(z)\,,
\end{eqnarray*}
and (\ref{eqrho6nc}) is proven.

At the same time, since $|\nabla\rho^{1/2}|\le |\nabla_Au|$, we have
\begin{equation*}
  \int_{Q(z)}|\nabla'\rho^{1/2}|^2 \,dx'\le 4E(z)\,.
\end{equation*}
Therefore
\begin{eqnarray*}
  E(z)&\gsim&\int_{Q(z)}\left[|\nabla'\rho^{1/2}|^2 + \kappa^2 \rho(1-\rho)^2\right]dx' \\
&\ge&
2 \int_{Q(z)} \kappa \rho^{1/2}(1-\rho) |\nabla'\rho^{1/2}| \,dx'
= \frac{\kappa}{2} \int_{Q(z)} |\nabla' \chi|\,dx'\,,
\end{eqnarray*}
since $\nabla'\chi = 4 \rho^{1/2}(\rho-1)\nabla'\rho^{1/2}$. This proves
(\ref{eqnablaprimechi}). 

Let now $\chi_\ell$ be defined as in  (\ref{eqdefaverage}). By Jensen's inequality and the mean-value
theorem we obtain
\begin{equation*}
  \int_{Q(z)} |\chi_\ell-\chi|\,dx' \le 
  \sup_{|h'|\le \ell} 
  \int_{Q(z)} |\chi(x'+h')-\chi(x')|\,dx' \le
  \ell \int_{Q(z)}  |\nabla'\chi|\,dx'\,.
\end{equation*}
Notice that this still holds if $\ell>L$ (indeed, in this case the coefficient could be improved to $L$).
This concludes the proof of (\ref{eqchichiell}).  
\end{proof}

\subsubsection{Equidistribution of the phases}
We show that in every section $Q(z)$ with a good energy bound the volume fraction of the normal
 phase is approximately ``right'', in the sense that it can be obtained assuming that $B_3$ 
 equals $\kappa/\sqrt2$ in the normal phase, and zero outside.
\begin{lemma}\label{lemmaequidistribution}
Let $(u,A)$ be admissible, and assume (\ref{eqparameters}) and $\rho\le 1$.
\begin{enumerate}
 \item  If
\begin{equation}\label{eqEzgoodequip}
E(z)\le \frac{1}{8} \kappa \bext L^2\,,
\end{equation}
 then  
  \begin{equation*}
    \int_{Q(z)} \chi \,dx'=    \int_{Q(z)} (1-\rho)^2 \,dx'\sim \frac{\bext}{\kappa}L^2\,. 
  \end{equation*}
\item If
\begin{equation}\label{eqEzgoodequiptot}
E\le \frac{1}{8} \kappa \bext L^2T\,,
\end{equation}
 then  
  \begin{equation*}
    \int_{Q_{L,T}} \chi \,dx=    \int_{Q_{L,T}} (1-\rho)^2 \,dx\sim \frac{\bext}{\kappa}L^2T\,. 
  \end{equation*}
   \end{enumerate}
\end{lemma}
\begin{proof}
  We start with assertion (i). Using Lemma \ref{lemmaHmeno12} 
  we write
  \begin{alignat*}1
\bext L^2 -  \frac{\kappa}{\sqrt2}  \int_{Q(z)}\chi\,dx' &=
    \int_{Q(z)} \left(B_3 - \frac{\kappa}{\sqrt2} (1-\rho)^2 \right)\,dx'\\
&=
\int_{Q(z)} (1-\rho)\left(B_3 - \frac{\kappa}{\sqrt2} (1-\rho)\right)\,dx'
+\int_{Q(z)} \rho B_3\,dx'
  \end{alignat*}
and estimate, using Cauchy-Schwarz, (\ref{eqsectionrhohunemnorhodue}), and $\kappa\le 1/2$, 
\begin{alignat*}1
  \left|\bext L^2 -  \frac{\kappa}{\sqrt2}  \int_{Q(z)} \chi \, dx'\right|&\le 
 \left( \int_{Q(z)} (1-\rho)^2\,dx'\right)^{1/2} E^{1/2}(z) +8E(z)
\\
&\le \frac14\kappa \int_{Q(z)} \chi\,dx'+ \frac{E(z)}{\kappa}+4 \frac{E(z)}{\kappa}\,.
\end{alignat*}
The conclusion follows.

The second part is proven analogously, just extending all integrals to $Q_{L,T}$.
\end{proof}

\subsubsection{Interior term}
The next lemma is the key ingredient of our proof. It shows that, if the  localization is performed
appropriately, in low-energy sections the field $B_3$ necessarily concentrates. The concentration is 
made quantitative by a test function constructed via Lemma  \ref{lemmatestfunction}
starting from $\chi=(1-\rho)^2$.
\begin{lemma}\label{lemmainterior}
For any
admissible $(u,A)$ with $\rho\le 1$, any parameters which obey (\ref{eqparameters}),
any $r\ge \ell>0$,
 and any $z\in(0,T)$ such that (\ref{eqEzgoodequip})  holds one has
 \begin{align*}
   \frac\kappa{\sqrt2} \int_{Q(z)}  \chi \,dx' - \int_{Q(z)} B_3 \psi \,dx'\lsim \frac{E(z)}{\kappa}+\ell  E(z)+
 \left(\frac{r^2\bext L^2 }{\ell^2\kappa}\right)^{1/2} E^{1/2} (z)\,.
\end{align*}
The function $\psi$ is the one
is obtained via Lemma \ref{lemmatestfunction} from 
the restriction to $Q(z)$ of $\chi=(1-\rho)^2$.
\end{lemma}
\begin{proof}
We first write
\begin{align}\nonumber
     \kappa \int_{Q(z)}  \chi \,dx'
  &= \kappa \int_{Q(z)} (\chi-\chi^2) dx' + \kappa\int_{Q(z)} \chi(\chi-\chi_\ell) dx'
  + \kappa\int_{Q(z)} \chi \chi_\ell dx'\,.
\end{align}
We now estimate the three terms on the right-hand side.
For the first one, we compute
\begin{eqnarray*}
   \kappa \int_{Q(z)}(  \chi-\chi^2) \,dx'&\le&
 \kappa \int_{Q(z)}  (1-\rho)^2(2\rho-\rho^2)\,dx'\\
&\le&  2 \kappa \int_{Q(z)}  \rho(1-\rho)^2\,dx'\lsim  \frac{E(z)}{\kappa}
\end{eqnarray*}
by (\ref{eqrho6nc}). 
For the second, 
\begin{equation*}
 \kappa\int_{Q(z)} \chi(\chi-\chi_\ell) dx' \le 
\kappa\int_{Q(z)} |\chi-\chi_\ell| dx' 
 \lsim\ell  E(z)
 \end{equation*}
by (\ref{eqchichiell})\,.
For the third, we write, recalling Lemma \ref{lemmatestfunction}(i) and the definition of $\chi$ and
$\chi_\ell$,
\begin{alignat*}1 
 \kappa\int_{Q(z)} \chi \chi_\ell dx'
 &\le \kappa\int_{Q(z)} (1-\rho) \psi dx'\,.
\end{alignat*}
Therefore
\begin{align}\label{eqchiurho2}
  \kappa \int_{Q(z)}  \chi \,dx'&\le \kappa\int_{Q(z)} (1-\rho) \psi\, dx'+c\frac{E(z)}{\kappa}+c\ell  E(z)\,.
\end{align}
At this point we write
\begin{equation*}
  \int_{Q(z)} B_3 \psi \,dx'= 
 \int_{Q(z)} \frac{\kappa}{\sqrt2} (1-\rho) \psi\,dx'+
 \int_{Q(z)}\left(B_3- \frac{\kappa}{\sqrt2} (1-\rho)\right) \psi\,dx'\,.
\end{equation*}
We use \begin{equation*}
  \left| \int_{Q(z)}\left(B_3- \frac{\kappa}{\sqrt2} (1-\rho)\right) \psi\,dx'
\right|   \le\|\psi\|_{L^2(Q_L)} E^{1/2} (z)
\end{equation*}
and (\ref{eqchiurho2}) to obtain
\begin{align*}
   \frac\kappa{\sqrt2} \int_{Q(z)}  \chi \,dx' \le \int_{Q(z)} B_3 \psi \,dx'+c\frac{E(z)}{\kappa}+c\ell  E(z)+
  \|\psi\|_{L^2(Q_L)} E^{1/2} (z)\,.
\end{align*}
Recalling  (\ref{eqlemmatestfunctl2}),
\begin{equation*}
\|\psi\|_{L^2(Q_L)}^2 
\lsim
\frac{r^2}{\ell^2}{\|\chi\|_{L^1}(Q_L)}\lsim
\frac{r^2\bext L^2 }{\ell^2\kappa}\,,
\end{equation*}
where we used once again Lemma \ref{lemmaequidistribution}(i),  concludes the proof.
\end{proof}

\subsubsection{Transport term}
It remains to relate the behavior of $B_3$ in the interior with the behavior 
at the boundary.
As in the sharp-interface case, this is done in two steps, but both steps are different than
the corresponding ones in Section \ref{secoldtransport}.
The first estimate is easy, but does not give the optimal bound, since 
it is oblivious to  the fact that $B'$ needs to be concentrated on a small volume
(this is relevant, since we are estimating an $L^1$ term with an energy that contains the corresponding $L^2$ norm). The estimate is then improved in the following lemma.
\begin{lemma}\label{lemmatransportshort}
Let $(u,A)$ be admissible, $\psi\in W^{1,2}_\loc(\R^2)$, $Q_L$-periodic.
  For any pair $z_1, z_2\in (0,T)$ we have
  \begin{equation*}
    \left|\int_{Q(z_1)}B_3\psi \,dx'- \int_{Q(z_2)}B_3\psi\,dx' \right|
    \le  |z_2-z_1|^{1/2} E^{1/2} \|\nabla'\psi\|_{L^2(Q_L)}\,.
  \end{equation*}
\end{lemma}
\begin{proof} We can assume without loss of generality that  $z_1<z_2$.
  We compute, using $\Div B=0$ and  the $Q_L$-periodicity of $B$ and
  $\psi$, 
  \begin{alignat*}1
     \int_{Q_L} [B_3(\cdot,z_2)-B_3(\cdot,z_1)]\psi \,dx'
=&\int_{Q_L\times(z_1,z_2)}
\frac{\partial B_3}{\partial x_3} \psi\,dx 
=-\int_{Q_L\times(z_1,z_2)}
\nabla'\cdot B' \psi\,dx\\ 
=&\int_{Q_L\times(z_1,z_2)}
 B'\cdot\nabla\psi\,dx\\ 
&\le\left(\int_{Q_{L}\times(z_1,z_2)} |B'|^2\,dx\right)^{1/2} 
\left(|z_2-z_1| \int_{Q_{L}} |\nabla\psi|^2 \,dx'\right)^{1/2}\\
&\le |z_2-z_1|^{1/2} E^{1/2} \|\nabla\psi\|_{L^2(Q_L)}\,.
  \end{alignat*}
This concludes the proof.
\end{proof}

\begin{lemma}\label{lemmatransportlong}
Let $(u,A)$ be admissible, and assume (\ref{eqparameters}) holds and $\rho\le 1$.
If (\ref{eqEzgoodequiptot}) holds, then for any $z\in (0,T)$ 
and any  $\psi\in W^{1,\infty}(\R^2)$, $Q_L$-periodic,
we have
  \begin{equation*}
    \left|\int_{Q(z)}B_3\psi\,dx' - \int_{Q(0)}B_3\psi \,dx'\right|
    \lsim
    \left(\frac{\bext L^2T}{\kappa}\right)^{1/2} 
\|\nabla'\psi\|_{L^\infty}E^{1/2}+ \|\nabla'\psi\|_{L^2(Q_L)} E^{1/2} \,.
  \end{equation*}
\end{lemma}
\begin{proof}
If $z\le 2$, this follows immediately from Lemma
  \ref{lemmatransportshort}. Assume $z>2$, and 
  fix $\delta\in(0,z/2)$. Let $\eta:\R\to\R$ be defined by
  \begin{equation*}
    \eta(x_3):=
    \begin{cases}
      \displaystyle
      \frac{x_3}{\delta} & \text{ if } 0<x_3<\delta\,,\\
      \displaystyle
      1 & \text{ if } \delta\le x_3\le z-\delta\,,\\
      \displaystyle
      \frac{z-x_3}{\delta} & \text{ if } z-\delta<x_3<z\,,\\
      0 & \text{ otherwise.}
    \end{cases}
  \end{equation*}
  We compute
  \begin{alignat*}1
    \int_{Q_{L,T}} B_3(x) &\psi(x') \frac{d\eta}{d x_3}(x_3) \,dx= 
    -\int_{Q_{L,T}} \frac{\partial B_3}{\partial x_3}(x) \psi(x')
    \eta(x_3) \,dx\\
=& \int_{Q_{L,T}} \nabla'\cdot B' \psi    \eta \,dx
= -\int_{Q_{L,T}}  B' \cdot (\nabla' \psi) \eta\,dx\\
=& -\int_{Q_{L,T}}  \rho B' \cdot (\nabla' \psi) \eta\,dx
- \int_{Q_{L,T}}  (1-\rho) B' \cdot (\nabla' \psi) \eta\,dx\,.
  \end{alignat*}
The second term can be estimated by
\begin{alignat*}1
  &\hskip-1cm \left|\int_{Q_{L,T}}  (1-\rho) B' \cdot (\nabla' \psi) \eta\,dx\right|\\
&   \le  \left(  \int_{Q_{L,T}}  (1-\rho)^2\,dx\right)^{1/2}
  \left(  \int_{Q_{L,T}}  |B'|^2\,dx\right)^{1/2}
\|\nabla'\psi\|_{L^\infty}\,.
\end{alignat*}
Since  (\ref{eqEzgoodequiptot}) holds, Lemma \ref{lemmaequidistribution} gives
\begin{equation*}
 \int_{Q_{L,T}}  (1-\rho)^2\,dx \lsim \frac{\bext L^2T}{\kappa} 
\end{equation*}
and therefore 
\begin{equation*}
\left|   \int_{Q_{L,T}}  (1-\rho) B' \cdot (\nabla' \psi) \eta\,dx\right|
  \lsim \left(\frac{\bext L^2T}{\kappa}\right)^{1/2} E^{1/2}
\|\nabla'\psi\|_{L^\infty}\,.
\end{equation*}
For the first term we use Lemma \ref{lemmaformula} to obtain
\begin{eqnarray*}
\left|   \int_{Q_{L,T}}  \rho B' \cdot (\nabla' \psi) \eta \,dx\right|
  & \lsim&  \int_{Q_{L,T}}  \left(|\nabla_A u|^2+|\calD_Au|^2\right) 
  |\nabla' \psi| \eta\,dx \\
&&\hskip5mm
   +\left|  \int_{Q_{L,T}} ( \nabla\times j_A) \cdot (\nabla' \psi) \eta\,dx\right| \,.
\end{eqnarray*}
The part containing $j_A$ can be transformed according to 
\begin{equation*}
  \int_{Q_{L,T}} \eta\,  (\nabla\times j_A)\cdot\nabla'\psi \,dx=
  -\int_{Q_{L,T}} \nabla\eta\cdot j_A\times \nabla'\psi\,dx\,.
\end{equation*}
This is easily proven by integration by parts and using the symmetry of
the triple product. Since $\|\nabla'\psi \, \frac{d \eta}{dx_3} \|_{L^2(Q_{L,T})}=\|\nabla'\psi\|_{L^2(Q_L)}
\|\frac{d \eta}{dx_3} \|_{L^2((0,T))}$ and $\|j_A\|_{L^2(Q_{L,T})}\lsim E^{1/2}$,
 we obtain
\begin{equation*}
\left|   \int_{Q_{L,T}}  \rho B' \cdot (\nabla' \psi) \eta \,dx\right|
\lsim E \|\nabla'\psi\|_{L^\infty} +  E^{1/2}\|\nabla'\psi\|_{L^2(Q_L)}
\|\frac{d \eta}{dx_3}\|_{L^2((0,T))} \,.
\end{equation*}
Adding terms, and replacing $\eta$ by $-\eta$, we conclude that 
\begin{equation}\label{eqfaverage}
    \left|\int_0^T f \frac{d \eta}{dx_3} \, dx_3 \right|\lsim
     \left(\frac{\bext L^2T}{\kappa}\right)^{1/2} E^{1/2}
\|\nabla'\psi\|_{L^\infty}+  E^{1/2}\|\nabla'\psi\|_{L^2(Q_L)}
\|\frac{d \eta}{dx_3}\|_{L^2((0,T))}  
\end{equation}
where we dropped the term $ E \|\nabla'\psi\|_{L^\infty}  $ using $\kappa\le 1$ and the assumption 
 (\ref{eqEzgoodequiptot}) and we defined
\begin{equation*}
  f(x_3):=\int_{Q(x_3)} B_3 \psi\,dx'\,.
\end{equation*}
This estimate controls the variation of $f$  on a scale $\delta$, and
can be combined with
Lemma \ref{lemmatransportshort}, which gives a bound on the H\"older $1/2$ norm
 of $f$, to obtain a pointwise estimate. Precisely, since
\begin{equation*}
  \int_0^T f \frac{d \eta}{dx_3} \,dx_3 = \frac1\delta\int_0^\delta f \,dx_3 - \frac1\delta
  \int_{z-\delta}^{z} f\,dx_3\,,
\end{equation*}
we get
\begin{equation*}
  |f(z)-f(0)|\le \left|\int_0^T f \frac{ d \eta}{dx_3} \,dx_3\right| +
   \sup_{x_3\in(0,\delta)} |f(x_3)-f(0)|+
   \sup_{x_3\in(z-\delta,z)} |f(x_3)-f(z)|\,.
\end{equation*}
Therefore
\begin{alignat*}1
\left|f(z)-f(0)\right|
   \lsim &
     \left(\frac{\bext L^2T}{\kappa}\right)^{1/2} E^{1/2}
\|\nabla'\psi\|_{L^\infty}+  E^{1/2}\frac{\|\nabla'\psi\|_{L^2(Q_L)}}{\delta^{1/2}}\\
&  +\delta^{1/2}E^{1/2}\|\nabla'\psi\|_{L^2(Q_L)}\,.
 \end{alignat*}
We finally choose $\delta=1$ and conclude the proof.
\end{proof}

\subsubsection{Exterior term}
\begin{lemma}\label{lemmaexterior}
  For all admissible $(u,A)$ we have
  \begin{equation*}
\left|    \int_{Q(0)} (B_3-\bext)\psi \,dx'\right|\lsim E^{1/2}    \|\psi\|_{H^{1/2}(Q_L)} \,.
  \end{equation*}
\end{lemma}
\begin{proof}
This is the same as in the sharp-interface case, cf. Lemma \ref{lemmaextsharp}.
\end{proof}

\subsubsection{Proof of the lower bound}
\begin{theorem}\label{theolowerbGL}
Let  $(u,A)$ be an admissible pair.
If 
\begin{equation*}
  \bext\le\frac18 \kappa\,,\hskip1cm
  \kappa\le\frac12\,,\hskip1cm\text{ and}\hskip1cm
  \kappa T\ge 1\,,
\end{equation*}
one has
\begin{equation*}
  E[u,A]\gsim \min \left\{\bext \kappa^{3/7} T^{3/7} L^2, 
\bext^{2/3} \kappa^{2/3}  T^{1/3}L^2\right\} \,.
\end{equation*}
\end{theorem}
\begin{proof}
By Lemma \ref{lemmarhominoredi1} we can assume without loss of generality that $\rho\le 1$.
We can assume $B-\Bext\in L^2(Q_L\times\R;\R^3)$, otherwise the energy is infinite, so that
we can use Lemma \ref{lemmaHmeno12}.

If (\ref{eqEzgoodequiptot}) does not hold then 
$E\gsim \bext L^2 (\kappa T)^{3/7}$,
since $\kappa T\ge 1$, and the proof is concluded. 
Therefore we can assume that (\ref{eqEzgoodequiptot}) holds.
We choose  $z\in(0,T)$ such that
$E(z)\le E[u,A]/T$, so that in particular (\ref{eqEzgoodequip}) holds.
Let $\psi$ be the function 
constructed via Lemma \ref{lemmatestfunction} from 
the restriction to $Q(z)$ of $\chi=(1-\rho)^2$, 
as in Lemma \ref{lemmainterior}, for some parameters $r,\ell$ still to be chosen. 
We start from the identity
\begin{equation*}
 \int_{Q(z)} B_3\psi \,dx'= \int_{Q_L}\bext\psi \,dx'  + \left[\int_{Q(z)}B_3\psi\,dx'
  -\int_{Q(0)}B_3\psi\,dx'\right]+ \int_{Q(0)} (B_3-\bext)\psi\,dx'\,.\end{equation*}
From Lemma \ref{lemmatransportlong}
  \begin{equation*}
    \left|\int_{Q(z)}B_3\psi\,dx' - \int_{Q(0)}B_3\psi \,dx'\right|
    \lsim
\frac1r    \left(\frac{\bext L^2T}{\kappa}\right)^{1/2} 
E^{1/2}+ \left(\frac{\bext L^2}{\kappa \ell^2}\right)^{1/2} E^{1/2} \,,
  \end{equation*}
  where we used $\|\nabla'\psi\|_{L^\infty}\lsim 1/r$ and 
  $\|\nabla'\psi\|_{L^2(Q_L)}\le 
   \left(\frac{\bext L^2}{\kappa \ell^2}\right)^{1/2}$ (both obtained 
   from Lemma \ref{lemmatestfunction}(iv) and (v) and  Lemma \ref{lemmaequidistribution}(i)).
Analogously, Lemma \ref{lemmaexterior} shows that
  \begin{equation*}
    \int_{Q(0)} (B_3-\bext)\psi \,dx'\lsim  \left(\frac{r\bext L^2}{\kappa\ell^2
    }\right)^{1/2} E^{1/2}\,,
  \end{equation*}
  where we inserted the  bound on the $H^{1/2}$ norm of $\psi$ : 
  \begin{eqnarray*}
    \|\psi\|_{H^{1/2}(Q_L)}&\le& \|\psi\|_{L^2(Q_L)}^{1/2} 
\,\, \|\nabla \psi\|_{L^2(Q_L)}^{1/2}
\lsim \left(\frac{r\bext L^2}{\kappa\ell^2 }\right)^{1/2}\,.
  \end{eqnarray*} 
Further, by Lemma \ref{lemmatestfunction}(iii),
\begin{equation*}
    \int_{Q_L}\bext\psi \,dx'\le \bext \|\psi\|_{L^1(Q_L)}\le 4\bext\frac{r^2}{\ell^2}
    \int_{Q(z)}  \chi \,dx'
    \,,
\end{equation*}
therefore
\begin{align*}
 \int_{Q(z)} B_3\psi \,dx'\le&  4\bext\frac{r^2}{\ell^2}
    \int_{Q(z)}  \chi \,dx'+
c\frac1r    \left(\frac{\bext L^2T}{\kappa}\right)^{1/2} 
E^{1/2}\\
&+ c\left(\frac{\bext L^2}{\kappa \ell^2}\right)^{1/2} E^{1/2}+c\left(\frac{r\bext L^2}{\kappa\ell^2
    }\right)^{1/2} E^{1/2}\,.
\end{align*}
At this point we use Lemma \ref{lemmainterior}, which states that
   \begin{align*}
   \frac\kappa{\sqrt2} \int_{Q(z)}  \chi \,dx'
\le   \int_{Q(z)} B_3 \psi \,dx'+ c\frac{E(z)}{\kappa}+c\ell  E(z)
+  c\left(\frac{r^2\bext L^2 }{\ell^2\kappa}\right)^{1/2} E^{1/2} (z)\,.
\end{align*}
  Combining the previous estimate gives
  \begin{align*}
     \frac\kappa{\sqrt2} \int_{Q(z)}  \chi \,dx'
     \le& 4\bext\frac{r^2}{\ell^2}  \int_{Q(z)}  \chi \,dx'
     +c \frac1r\left(\frac{\bext L^2T}{\kappa}\right)^{1/2} 
E^{1/2}+ c \left(\frac{\bext L^2}{\kappa \ell^2}\right)^{1/2} E^{1/2} \\
&  +c   \left(\frac{r\bext L^2}{
    \kappa\ell^2}\right)^{1/2} E^{1/2} 
     +c\frac{E(z)}{\kappa}+c\ell  E(z)+c
 \left(\frac{r^2\bext L^2 }{\kappa\ell^2}\right)^{1/2} E^{1/2} (z)\,.
  \end{align*}
Assume now that 
\begin{equation}\label{eqdefadmissrlc}
  0<\ell\le r\le \left(\frac{\kappa}{8\bext}\right)^{1/2} \ell\,,
\end{equation}
so that the first term on the right is no larger than the one on the left divided by $\sqrt2$. 
Recalling Lemma \ref{lemmaequidistribution}(i) and $E(z)\le E/T$ we conclude 
 that for all pairs $(r,\ell)$ which obey 
(\ref{eqdefadmissrlc}), we have
  \begin{align*}
  \bext L^2 \lsim&\frac1r\left(\frac{\bext L^2T}{\kappa}\right)^{1/2} 
E^{1/2}+  \left(\frac{\bext L^2}{\kappa \ell^2}\right)^{1/2} E^{1/2} \\
&  +  \left(\frac{r\bext L^2}{
    \kappa\ell^2}\right)^{1/2} E^{1/2} 
     +\frac{E}{\kappa T}+\ell  \frac ET+
 \left(\frac{r^2\bext L^2 }{\kappa\ell^2}\right)^{1/2} \frac{E^{1/2}}{T^{1/2}}\,.
  \end{align*}
We remark that only the second and the fourth term are new with respect to the sharp-interface case.
  At least one of the six terms has to be at least one-sixth of the total, therefore
\begin{equation*}
  E\gsim \min \left\{
  \bext L^2\frac{r^2 \kappa}{T}, 
    \bext L^2\kappa \ell^2  ,
  \bext L^2 \frac{\kappa\ell^2}{r}, 
  \bext L^2\kappa T, 
  \bext L^2\frac{T}{\ell},
  \bext L^2\frac{\kappa\ell^2 T}{r^2}
  \right\}\,.
\end{equation*}
Equivalently, 
\begin{equation*}
  E\gsim \kappa \bext L^2T \min\left\{
  \frac{r^2 }{T^2},
  \frac{\ell^2}{T},
  \frac{\ell^2}{rT},
  1,
  \frac{1}{\kappa\ell}, 
  \frac{\ell^2}{r^2}
\right\}.
\end{equation*}
The fourth term can be dropped, since the sixth one is always less than
1.  
Therefore we can focus on 
\begin{equation}\label{eqfinalboundFc}
  E\gsim \kappa \bext L^2T 
  \min\left\{
  \frac{r^2 }{T^2},
  \frac{\ell^2}{T},
  \frac{\ell^2}{rT},
  \cdot,
  \frac{1}{\kappa\ell}, 
  \frac{\ell^2}{r^2}
\right\}\,,
\end{equation}
where a dot marks the term we already know to be irrelevant. 
In comparing with (\ref{eqfinalboundF}), we see  that the only new term is the second one, $\ell^2/T$.
%
Averaging the first and the sixth we see that $\ell/T$
would be irrelevant; hence the second term is irrelevant for all
choices of $\ell\gsim 1$. We shall see later that this is the case.

We finally have to choose $r$ and $\ell$, and check that in each case
some terms give the optimal bound, and the others are
irrelevant. Since we already know the scalings, we do not need to
check all possible combinations. 
Balancing the  first, third and  fifth term suggests the choice
\begin{equation*}
  \ell=T^{4/7}\kappa^{-3/7}\,, \hskip5mm r=T^{5/7}\kappa^{-2/7}\,.
\end{equation*}
This choice is admissible only if  (\ref{eqdefadmissrlc}) is satisfied. The 
first condition is always true, since $\kappa T\ge1$; the second one is equivalent to
$(\kappa T)^{1/7}\le  (\kappa/8\bext)^{1/2}$. 
Since $\kappa \le 1 \le \kappa T$, one can compute $\ell=(\kappa T)^{4/7}/\kappa\ge 1$,
hence the second term can indeed be dropped. In this case,
(\ref{eqfinalboundFc}) becomes
\begin{equation*}
  E\gsim
 \kappa^{3/7}\bext L^2T^{3/7} \min\left\{
 1,
 \cdot, 
 1, 
 \cdot,
 1, 
 (\kappa T)^{2/7}
 \right\}\,.
\end{equation*}
Since $\kappa T\ge1$, in the regime $(\kappa T)^{1/7}\le  (\kappa/8\bext)^{1/2}$
we have shown $E\gsim \kappa^{3/7}\bext L^2T^{3/7}$.

If instead $(\kappa T)^{1/7}\ge (\kappa/8\bext)^{1/2}$, we need to choose
$r= \ell (\kappa/8\bext)^{1/2}$. Then (\ref{eqdefadmissrlc}) is always satisfied. Balancing the first and fifth term in 
  (\ref{eqfinalboundFc}) with this constraint results, after some
  rearrangement, into 
  \begin{equation*}
    \ell = \frac{\bext^{1/3}T^{2/3}}{\kappa^{2/3}}
  \end{equation*}
and
\begin{equation*}
  E\gsim
 \kappa^{2/3}\bext^{2/3}L^2T^{1/3} \min\left\{
1, 
\frac{T\kappa}{\kappa}  \frac{\bext}{\kappa},
(\kappa T)^{1/3} \left(\frac{\bext}{\kappa}\right)^{7/6},
\cdot,
1,
(\kappa T)^{2/3} \left(\frac{\bext}{\kappa}\right)^{4/3}
\right\}\,.
\end{equation*}
Again, a  dot marks terms we already know to be irrelevant. We observe that since $\kappa T\ge1$ and 
 $(\kappa T)^{1/7}\ge(\kappa/8\bext)^{1/2}$ one has
 \begin{equation*}
   (\kappa T)^{\alpha} \left(\frac{\bext}{\kappa}\right)^{\beta}
   \gsim (\kappa T)^{\alpha-\frac27\beta} \gsim 1
 \end{equation*}
whenever $\alpha\ge\frac27\beta$. This permits to show that
the second, third and sixth terms do not contribute, and
therefore concludes the proof.
\end{proof}

\section{Upper bound}\label{secupperb}
Before presenting our construction for the energy $\EGL$ we construct
fields with optimal scaling for the sharp-interface functional $F$. 
We prove a refined version of the results of \cite{ChoksiKohnOtto2004}, 
giving a construction which satisfies several additional
properties, which will be needed in the following generalization to $\EGL$.
In particular, we need to make sure that there is an integer number of 
flux quanta in each flux tube and that the thickening of the tubes
on the scale of the correlation length still has small volume. The construction
for $\EGL$ will then be derived from this one.

\subsection{Construction with sharp interfaces}
\label{secconstrsharp}
The key point in the construction is to use in two stages 
a subdivision of the domain into subsets with integer flux. This is done
via the domain subdivision algorithm of  Lemma \ref{lemmasubdivideflux}.
We first subdivide the domain down to scale $L/N$, 
and obtain rectangles $(r_j)_{j=1\,\dots, N^2}$, which will 
be used to fix  the microstructure in the central section of the sample, $Q_L\times \{T/2\}$. 
In each of these rectangles, the magnetic
field will be concentrated in a smaller concentric  rectangle  $\hat r_j$, keeping the same flux. 
Then we subdivide a second time, again using  Lemma \ref{lemmasubdivideflux}, 
down to smaller rectangles, which will be the ones used close to the surfaces of the sample,
$Q_L\times\{0\}$ and $Q_L\times\{T\}$.
At the same time,
the total flux inside each rectangle is concentrated in a smaller rectangle, so that
the intensity of the magnetic field is the one preferred by the energy, $\kappa/\sqrt2$ (see Figure \ref{figpropchilower2}).
After this setup we will make the actual branching construction, which corresponds
to the subdivision
generated in the second application of Lemma \ref{lemmasubdivideflux}.

Before stating the main result of this section
we recall the definition of $F$ in  (\ref{eqdefF})
and introduce the notation $(\omega)_\rho$ for a $\rho$-neighbourhood of a set $\omega$. Precisely, 
for
$\omega \subset\R^3$ and $\delta>0$,
\begin{equation}\label{eqdefomegarho}
(\omega)_\delta:=\{x\in\R^3: \dist(x,\omega)<\delta\}=\bigcup_{x\in \omega} B_\delta(x)
\end{equation}
where the distance is 
interpreted $Q_L$-periodically, $\dist(x,A):=\inf\{|x-z-kL|: z\in A, k\in \Z^2\times\{0\}\}$.
\begin{theorem}\label{theochiupper1}
For any $\bext,\kappa, L,T>0$ such that
\begin{equation*}
2\bext\le \kappa\le \frac12\,,\hskip1cm
  \kappa T\ge 1 \,,\hskip1cm\bext L^2\in 2\pi \Z\,,
\end{equation*}
and
\begin{equation}\label{eqLadmissible}
 L\ge \min\left\{
 \frac{8T^{2/3}}{(\kappa \bext)^{1/6}}, 
  \frac{8T^{4/7}\kappa^{1/14}}{ \bext^{1/2}}
 \right\}
\end{equation}
there is a pair $(\chi,B)$, admissible in the sense of Definition \ref{defadmisssharp},
and such that
\begin{equation*}
 F[\chi,B]\lsim \min \left\{\bext \kappa^{3/7} T^{3/7} L^2, 
\bext^{2/3} \kappa^{2/3}  T^{1/3}L^2\right\}\,.
\end{equation*}
The set $\omega:=\{x\in Q_{L,T}: \chi(x)=1\}$  is formed by the union of finitely many sheared
parallelepipeds, with two faces normal to $e_3$, and
\begin{equation}\label{eqestchiuppernbd}
 |(\omega)_{1/\kappa}\setminus\omega| \lsim \frac{1}{\kappa}\int_{Q_{L,T}} |D\chi|\,.
\end{equation}
For any $z\in (0,T)$  the flux of $B$ across each
connected section of $\{x_3=z\} \cap \omega$ is an integer multiple of $2\pi$.
\end{theorem}

\begin{figure}[t]
\begin{center}
 \includegraphics[width=8cm]{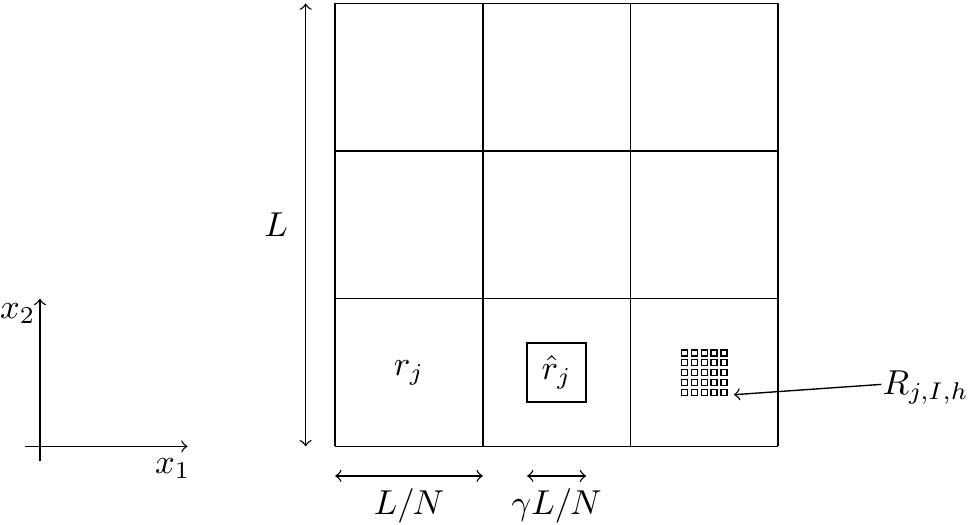} 
 \end{center}
\caption{Sketch of the notation used in the proof of Theorem \ref{theochiupper1}.
\label{figpropchilower2}}
 \end{figure}

Before giving the proof of Theorem \ref{theochiupper1}, we formulate and prove the partial results
that will be needed. We start with the domain subdivision.
Here we refine a flux pattern making sure that each component keeps the quantization condition. In order
for the field to maintain the optimal intensity, the areas are changed. Some parts of the construction
would be simpler if one would work with squares, but then the macroscopic distribution of the flux would be
modified. We work with rectangles, of aspect ratio uniformly close to 1, so that on any scale the perturbation
to the distribution of flux is kept to a minimum.
\begin{lemma}\label{lemmasubdivideflux}
 Let $R_{0,1}=(0,a)\times (0,b)$ and $B_\ast>0$ be such that $\frac 13 a\le b\le 3a$ and 
 $ab B_\ast\in 2\pi \Z$. 
 Then for any $k\in \N$ there are $4^k$ pairwise disjoint rectangles 
 $\{R_{k,i}\}_{i=1,\dots, 4^k}$, $R_{k,i}=x_{k,i}+(0,a_{k,i})\times(0, b_{k,i})$, 
 such that $\frac 13 a_{k,i}\le b_{k,i}\le 3a_{k,i}$, 
  each $R_{k,i}$ is the union of four
  $R_{k+1,i}$ (up to null sets), and
  $ a_{k,i}b_{k,i}B_\ast\in 2\pi\Z$ for all $k,i$.
 Further, $|a_{k,i}b_{k,i}-4^{-k}ab|\le 4\pi/B_\ast$.
\end{lemma}
\begin{proof}
 The lemma follows immediately from Lemma \ref{lemmasubdivideflux2}, if one ignores half of the stages.
 Precisely, if $(R^*_{k,i})_{i=1,\dots, 2^k}$ are the rectangles produced by  Lemma \ref{lemmasubdivideflux2}, 
 we set $R_{k,i}=R^*_{2k,i}$, for $i=1,\dots, 4^k=2^{2k}$.
\end{proof}
For the proof it is more convenient to focus on the following version, 
in which only one subdivision of each rectangle into two is performed at each step. To simplify
the notation, we say that a rectangle $R=x'+(0,a)\times (0,b)\subset\R^2$ is $B_\ast$-good, for some $B_\ast>0$, if
\begin{equation}
  \frac 13 a\le b\le 3a\hskip3mm \text{ and } \hskip3mm abB_\ast\in 2\pi \Z\,.
\end{equation}
\begin{lemma}\label{lemmasubdivideflux2}
 Let $B_*>0$, and let $R_{0,1}$ be a $B_\ast$-good rectangle.
 Then for any $k\in \N$ there are $2^k$ pairwise disjoint $B_\ast$-good rectangles 
 $\{R_{k,i}\}_{i=1,\dots, 2^k}$
 such that each $R_{k,i}$ is the union of two
  $R_{k+1,i}$ (up to null sets) and $||R_{k,i}|-2^{-k}|R_{0,1}||\le 4\pi/B_\ast$.
\end{lemma}
\begin{proof}
The construction is iterative, starting with $R_{0,1}$.
Consider one rectangle at step $k$, say,  $R_{k,i}=x_{k,i}+(0,a_{k,i})\times(0, b_{k,i})$.
Assume for definiteness that $a_{k,i}\le b_{k,i}$. 
If $R_{k,i}$ is empty, i.e., has side lengths zero,  we replace $R_{k,i}$ by two empty rectangles, 
setting $R_{k+1,2i}:=R_{k+1,2i+1}:=\emptyset$.
If $B_\ast a_{k,i}b_{k,i}=2\pi$, we replace it by a copy of itself and an empty rectangle,
setting $R_{k+1,2i}:=R_{k,i}$ and $R_{k+1,2i+1}:=\emptyset$.
Otherwise, we set 
\begin{equation*}
 y:=\frac{2\pi}{B_\ast a_{k,i}} \left\lfloor \frac{B_\ast a_{k,i}b_{k,i}}{4\pi}\right\rfloor
 \end{equation*}
and replace $R_{k,i}$ by the two rectangles
$x_{k,i}+(0,a_{k,i})\times (0,y)$ and 
$x_{k,i}+(0,a_{k,i})\times (y,b_{k,i})$. Here
$\lfloor t\rfloor=\max\{z\in\Z: z\le t\}$.

Clearly $y\le b_{k,i}/2$. At the same time, since for any $z\in\N$ with $z\ge 2$ one has
$\lfloor z/2\rfloor \ge z/3$, one has $y\ge b_{k,i}/3$.
Therefore the aspect ratio of the two new rectangles is also not larger than 3 and they are
$B_\ast$-good.

It remains to estimate the area.
Consider one rectangle $R_{k,i}$ at stage $k$. It has been generated
from $R_{0,1}$ by  $k$ subdivision steps. Let $s_h$ be the area of the
rectangle at stage $h$ along this subdivision path, so that $s_0:=|R_{0,1}|$. By the definition of $y$, 
we obtain $|2s_h-s_{h-1}|\le 4\pi/B_\ast$. Summing the series we obtain
\begin{equation*}
 |2^ks_k-|R_{0,1}|| \le \sum_{h=1}^k |2^h s_h-2^{h-1} s_{h-1}|
 \le \frac{4\pi}{B_\ast} \sum_{h=1}^k 2^{h-1} 
\le 2^k\frac{4\pi}{B_\ast} \,,
 \end{equation*}
which gives $|s_k-2^{-k}|R_{0,1}||\le 4\pi/B_\ast$.
\end{proof}

We next estimate the energy of a flux configuration on the boundary. We assume that the magnetic
field inside each rectangle $R_j$ is concentrated in a subrectangle $r_j$. The difference
of the fields $a_j\uno_{r_j}-A_j\uno_{R_j}$ then has average zero over the larger rectangle $R_j$.
Here and below
 we denote by $\uno_E$ the characteristic function of a set $E$,
$\uno_E(x)=1$ if $x\in E$, $0$ otherwise. 
\begin{lemma}\label{lemmaestimatehminus12}
For $M\in\N$ let
\begin{equation*}
 \{R_1,\dots,  R_{M}\}\hskip5mm
 \text{ and }\hskip5mm
\{r_1,\dots,  r_{M}\}
 \end{equation*}
be rectangles
 such that  each of them has aspect ratio not larger than 3, $r_j\subset R_j\subset Q_L$, 
with the $R_j$ pairwise disjoint.
Let $a_j, A_j\in\R$ be such that $a_j| r_j|=A_j|R_j|$.
Then
\begin{equation*}
 \left\| \sum_{j=1}^M  a_j \uno_{r_j}-A_j\uno_{R_j}\right\|_{H^{-1/2}(Q_L)}^2 \lsim \sum_{j=1}^M a_j^2 \, |r_j|^{3/2}\,.
\end{equation*}
\end{lemma}
\begin{proof}
Fix one index $j$, and let $g:=a_j \uno_{r_j} -A_j\uno_{R_j}$. 
Since $g$ has average 0 over $R_{j}$, for any $\varphi\in L^1(R_{j})$ we have
\begin{equation*}
\left|  \int_{R_{j}} g \varphi\, dx' \right|=
 \left| \int_{R_{j}} g (\varphi-\varphi_0) dx'\right| =
  \left|\int_{r_j} a_j (\varphi-\varphi_0) dx' \right|
  \le |a_j| \, \|\varphi-\varphi_0\|_{L^1(r_j)}\,,
  \end{equation*}
  where $\varphi_0$ is the average of $\varphi$ over $R_{j}$.
  Since the trace of a $H^1$ function (in three dimensions) belongs to $L^4$, 
  if $\varphi\in H^1_\loc(R_j\times(0,\infty))$ we obtain
\begin{equation*}
\|\varphi-\varphi_0\|_{L^1(r_j)}\le
| r_j|^{3/4}\|\varphi-\varphi_0\|_{L^4( r_j)}\lsim
| r_j|^{3/4}\|\nabla \varphi\|_{L^{2}(R_{j}\times(0,\infty))} \,.
\end{equation*}
Since the aspect ratio of the rectangles is controlled, the constant is universal.

Now fix $\Phi\in H^{1}_\loc(Q_L\times(0,\infty))$, $Q_L$-periodic. Let $\varphi_j$ be the average of $\Phi$ over $R_j$. Then the 
same computation gives
\begin{align*}
\left|  \int_{Q_L}\sum_j \left( a_j \uno_{r_j}-A_j\uno_{R_j}\right)\Phi\, dx' \right|
&\le \sum_j   \left|\int_{r_j} a_j (\Phi-\varphi_j) dx' \right|
  \lsim\sum_j |a_j| \, |r_j|^{3/4} \|\nabla \Phi\|_{L^{2}(R_{j}\times(0,\infty))}\\
&  \le \left(\sum_j a_j^2 \, |r_j|^{3/2}\right)^{1/2}  \|\nabla \Phi\|_{L^{2}(Q_L\times(0,\infty))}\,,
  \end{align*}
where in the last step we used Cauchy-Schwarz.
\end{proof}
Controlling the $H^{-1/2}$ norm of the normal component of $B$
on the boundary is sufficient to estimate the energetic cost of the
magnetic field outside the sample. We recall this general fact in the following Lemma.
\begin{lemma}\label{lemmah12}
 Let $g\in L^2_\loc(\R^2)$, $Q_L$-periodic, with average $\bext\in\R$. Then there is $B\in L^2_\loc(\R^3;\R^3)$,
 also $Q_L$-periodic, such that $\Div B=0$, $B_3(x',x_3)=g(x')$ for $x_3\ge 0$, and
 \begin{equation*}
  \int_{Q_L\times(-\infty,0)} |B-\bext e_3|^2dx\lsim \|g-\bext\|_{H^{-1/2}(Q_L)}^2\,.
 \end{equation*}
\end{lemma}
\begin{proof}
 It suffices to consider the case $\bext=0$. Let $\hat g(k')$ be the Fourier coefficients of $g$,
 so that $g(x')=\sum_{k'\in 2\pi\Z^2/L} e^{ik'\cdot x'} \hat g(k')$
 and $\sum_{k'\ne0} |\hat g|^2(k')/|k'|\sim \|g\|_{H^{-1/2}}^2$.
 Since $g$ has average $0$, $\hat g(0)=0$.
 We define, for $x_3\le0$,
 \begin{equation*}
  B(x',x_3):=\sum_{k'\in 2\pi\Z^2/L} e^{ik'\cdot x'} \hat B(k',x_3)
 \end{equation*}
where
 \begin{equation*}
 \hat B_3(k',x_3):=\hat g(k') e^{|k'|x_3}\,, \hskip5mm
 \hat B'(k',x_3):=i\frac{k'}{|k'|} \hat g(k') e^{|k'|x_3}\,. 
 \end{equation*}
  It is then straightforward to check that the stated properties are satisfied.
\end{proof}

Before starting the construction in the interior region, we introduce a separate notation for the interior contribution to the energy.
We define, for $\Omega\subset\R^3$ open, $\chi\in BV(\Omega)$ and $B\in L^2(\Omega;\R^3)$,
\begin{equation}\label{eqdefFint}
  F^\Int[\chi,B,\Omega] := \int_{\Omega}  \kappa |D\chi|  +\int_{\Omega} \left[|B'|^2 + \chi\left(B_3
  - \frac{\kappa}{\sqrt2}\right)^2 \right]dx,
    \end{equation}
so that $F[\chi,B]=F^\Int[\chi,B,Q_{L,T}]+ \int_{Q_L\times[\R\setminus(0,T)]} |B-\bext e_3|^2 dx$.
By Lemma \ref{lemmah12} it suffices to control 
$F^\Int[\chi,B,Q_{L,T}]+\|B-\bext\|_{H^{-1/2}(Q_L\times\{0\})}^2+\|B-\bext\|_{H^{-1/2}(Q_L\times\{T\})}^2$.

The next construction step is a procedure to generate an admissible $\chi$ and $B$  with given boundary data 
in a slab $\R^2\times (0,t)$. The explicit construction is done for the case that the boundary data
are characteristic functions of rectangles. An extension to the case where circles are used,
which gives a smaller surface energy  (by a factor which does not affect the scaling), 
is discussed in \cite{ContiGoldmanOttoSerfaty}. 
\begin{lemma}\label{lemmatransformrectintorect}
 Let $r:=p+(0,a)\times (0,b)$,
 $\hat r:=\hat p+(0,\hat a)\times (0,\hat b)\subset\R^2$ be two rectangles with $|r|=|\hat r|$, $a\sim b$ and $\hat a\sim \hat b$.
 For any $t\gsim a$ there are $(\chi,B):\R^2\times[0,t]\to\{0,1\}\times \R^3$ such that
 \begin{equation*}
  B_3(\cdot, \cdot, 0)=\frac\kappa{\sqrt2}\uno_r\,,\hskip1cm
 B_3(\cdot, \cdot, t)=\frac\kappa{\sqrt2}\uno_{\hat r}\,\,,\hskip1cm \Div B=0\,,
 \end{equation*}
 the first two in the sense of traces, and such that, defining $\omega:=\{x: B(x)\ne 0\}$,
\begin{equation*}
  B_3=\frac\kappa{\sqrt2}\chi=\frac\kappa{\sqrt2}\uno_\omega \,,\hskip3mm 
   F^\Int[\chi,B,\R^2\times(0,t)]\lsim  
   \kappa a t+   \kappa a |p-\hat p|+\kappa^2
   \frac{|p-\hat p|^2+ab}{t} ab\,.
\end{equation*}
If $R=(x_1,y_1)\times(x_2,y_2)$ is such that $r\cup\hat r\subset R$, 
then $\omega\subset R\times[0,t]$. 
For any $\delta\lsim a$,
\begin{equation*}
 |(\omega)_\delta\setminus \omega|\lsim \delta a t
\end{equation*}
with $(\omega)_\delta$ defined as in (\ref{eqdefomegarho}).
\end{lemma}
\begin{proof} 
 We define $\varphi:\R\to\R$ by
 \begin{equation*}
  \varphi(x_3):= \exp\left( \frac{x_3}{t} \ln\frac{\hat a}{a}\right)
= \left(\frac{\hat a}{a}\right)^{x_3/t}\,.\end{equation*}
Then $\varphi(0)=1$, $\varphi(t)=\hat a/a=b/\hat b$, $|\varphi'|(x_3)\lsim 1/t$ for $x_3\in[0,t]$.
Further, we define
\begin{equation*}
 v(x):=
  \frac{x_3}{t}
  \begin{pmatrix}
 \hat p_1-p_1\\ \hat p_2-p_2\\0  
  \end{pmatrix}
+    
  \begin{pmatrix}
  x_1\varphi(x_3) \\ x_2/\varphi(x_3) \\ x_3
      \end{pmatrix}\,,
\end{equation*}
which is a diffeomorphism of $\R^2\times[0,t]$ into itself, with $\det \nabla v=1$ pointwise,
and finally  define $\chi$ and $B$ by
\begin{equation}\label{eqdefBv}
\chi(v(x)):=\uno_r(x')\,, \hskip2mm
B_3(v(x)):=\frac\kappa{\sqrt2}\uno_r(x')\,, \hskip2mm
B'(v(x)):=\frac\kappa{\sqrt2}\uno_r(x')\partial_3 v'(x)\,.
\end{equation}
Let $\theta\in C^1_c(\R^2\times (0,t))$. By a change of variables
\begin{equation*}
  \int_{\R^2\times(0,t)} (\partial_3\theta\, B_3 + \nabla'\theta \cdot B' )dx
  =\int_{\R^2\times(0,t)} (\partial_3\theta \circ v \, B_3\circ v + \nabla'\theta \circ v \cdot B'\circ v )dx\,.
\end{equation*}
Inserting the definition from (\ref{eqdefBv}) this becomes
\begin{equation*}
 \frac\kappa{\sqrt2} \int_{\R^2\times (0,t)}  [\partial_3\theta \circ v \,\uno_r+ \nabla'\theta \circ v\cdot (\uno_r\, \partial_3 v') ]dx
  =\frac\kappa{\sqrt2}\int_{\R^2\times (0,t)} \uno_r(x')\frac{d}{dx_3} (\theta\circ v)(x)dx=0\,.
\end{equation*}
Therefore the divergence condition is satisfied.

Finally, $\|B'\|_{L^\infty}\le \kappa \|\uno_r\partial_3 v'\|_{L^\infty}\le \kappa |p-\hat p|/t + c\kappa a/t$. Since $a\sim b$ and 
the volume of its support is $tab$, we obtain
\begin{equation*}
  \int_{\R^2\times (0,t)} |B'|^2 dx \lsim\kappa^2\frac{|p-\hat p|^2+ab}{t} ab\,.
\end{equation*}
From the definition of $\omega$ one easily obtains the other properties.
\end{proof}

At this point we present the branching construction, which gives the refinement  of the magnetic flux close to the boundary.
We start from a prescribed interior structure, which can be obtained either by uniform subdivision of the
domain and quantization of the field, or by nonuniform subdivision of the domain with a uniform field. 
We stress
that the rectangles in which the construction is localized do not need to cover $Q_L$. Indeed, for very small fields
the optimal scaling is only obtained if these rectangles cover a very small fraction of $Q_L$ (the volume fraction will be $\gamma^2$ in the  proof
of Theorem \ref{theochiupper1} below), see Figure \ref{figinternal}.
\begin{figure}
 \begin{center}
  \includegraphics[height=6cm]{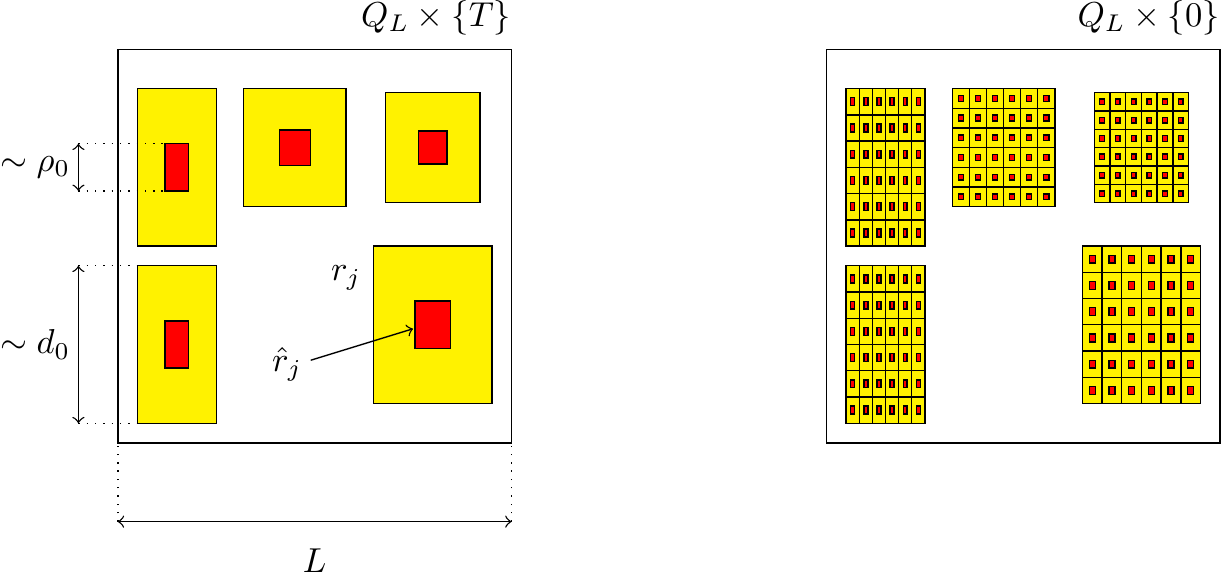}
\end{center}
\caption{Sketch of the horizontal geometry in Lemma 
\ref{lemmaconstrinternal}. Left panel: cross-section at $x_3=T$, with the 
initial rectangles $r_j$ (yellow) and $\hat r_j$ (red) shown. 
Right panel: cross-section at $x_3=0$, with the finer structure of $B_3(\cdot,0)$ shown.}\label{figinternal}
\end{figure}

\begin{lemma}\label{lemmaconstrinternal}
Let $\kappa,L,T,  d_0,\rho_0,N>0$, with $N^2\in\N$ and $\kappa\le 1$, and assume that two sets of $N^2$ rectangles
are given,
\begin{equation*}
 \{r_1,\dots,  r_{N^2}\}\hskip5mm
 \text{ and }\hskip5mm
\{\hat r_1,\dots,  \hat r_{N^2}\}\,,
 \end{equation*}
such that  each of them has aspect ratio no larger than 3, $\hat r_j\subset r_j\subset Q_L$, 
the $r_j$ are pairwise disjoint, 
$|\hat r_j|\sim \rho_0^2$, 
$|r_j|\sim d_0^2$, with
\begin{equation*}
\frac\kappa{\sqrt2}| \hat r_j|\in 2\pi \N \hskip5mm\text{ for all $j$. }
\end{equation*}
Assume 
\begin{equation}\label{eqkapparho0}
\kappa\rho_0\gsim 1  \,.
\end{equation}
 Then
there is a pair $(\chi,B)$, admissible in the sense of
Definition \ref{defadmisssharp},
and such that
\begin{equation*}
 B_3(\cdot,T)=
  \frac\kappa{\sqrt2} \sum_{j=1}^{N^2} \uno_{\hat r_j}\hskip1cm\text{ on $Q_L$}\,,
\end{equation*}
\begin{equation}\label{eqfintpropcostr}
F^\Int[\chi,B,Q_{L,T}]
  \lsim  \kappa \rho_0 TN^2+\kappa^2 \frac{\rho_0^2 d_0^2N^2}{T}
\end{equation}
and, with $\hat b_j:=\kappa/\sqrt2  (|\hat r_j|/|r_j|)$, 
\begin{equation}\label{eqhm12QLd}
  \left\|B_3(\cdot, 0)-\sum_{j=1}^{N^2} \hat b_j \uno_{r_j} 
  \right\|_{H^{-1/2}(Q_L)}^2 
  \lsim  {\kappa\rho_0^2N^2}  + \frac{\kappa^{2}\rho_0^3N^2d_0}{T}  \,.
 \end{equation}
The set $\omega:=\{x\in Q_{L,T}: \chi(x)=1\}$ is formed by the union of finitely many sheared
parallelepipeds, with two faces normal to $e_3$, and
\begin{equation}\label{eqestchiuppernbd2}
 |(\omega)_{1/\kappa}\setminus\omega| \lsim \frac{1}{\kappa}\int_{Q_{L,T}} |D\chi|\,.
\end{equation}
For any $z\in (0,T)$  the flux of $B$ across each
connected component of $\{x_3=z\} \cap \omega$ is an integer multiple of $2\pi$.
\end{lemma}

\begin{proof}
We first observe that a simple construction which obeys all kinematic constraints is obtained using
a pattern which does not depend on $x_3$. From the boundary data at $x_3=T$ we see that this necessarily is
\begin{equation*}
\hat B=(0,0,\frac{\kappa}{\sqrt2}\hat\chi)\,,\hskip1cm \hat\chi(x',x_3)=\sum_{j=1}^{N^2} \uno_{\hat r_j}(x')\,.
\end{equation*}
A simple computation shows that
\begin{equation*}
 F^\Int[\hat\chi,\hat B, Q_{L,T}] \lsim \kappa\rho_0 TN^2
\end{equation*}
and, using Lemma \ref{lemmaestimatehminus12},
\begin{equation*}
\left\|\hat B_3(\cdot,0)-\sum_{j=1}^{N^2} \hat b_j \uno_{r_j} 
  \right\|_{H^{-1/2}(Q_L)}^2 
  \lsim \sum_j \kappa^2|\hat r_j|^{3/2} \lsim \kappa^2N^2\rho_0^3\,.
\end{equation*}
At the same time $|(\omega)_{1/\kappa}\setminus\omega|\lsim N^2T (\rho_0/\kappa+1/\kappa^2)$.
If  $T\lsim d_0$, the proof is concluded. 
We observe that if $\kappa\rho_0\lsim1$ the energy estimates would also hold, but not the one
on the measure of $(\omega)_{1/\kappa}$. 
In the following we assume  $T\gg d_0$.

We choose $I\in\N$ such that
\begin{equation}\label{eqconstrIcond}
2^I\sim  \min\left\{\kappa\rho_0, \frac{T}{d_0}\right\}\,.
\end{equation}
Possibly reducing $I$ by a few units, which does not affect the statement, we can assume that
\begin{equation}\label{eq4iadm}
 4^I \le \frac{\hat b_j|r_j|}{8\pi} = \frac{\kappa|\hat r_j|}{8\pi \sqrt2}\hskip5mm \text{ for all $j$}
\end{equation}
(to see this, one observes that $\kappa|\hat r_j|\sim\kappa\rho_0^2>(\kappa\rho_0)^2$).
The construction is performed independently in each set $r_j\times (0,T)$,
we take both fields to  vanish outside the union of these sets.
Let $R_{j,i,h}$ be the 
rectangles given in Lemma \ref{lemmasubdivideflux} for $1\le i\le I$, $1\le h\le 4^i$,
starting from $r_j$, using the  field $B_\ast=\hat b_j$. We denote by $i$ the refinement
level, by $h$ the numbering of the rectangles at each level.
By construction we have $||R_{j,i,h}|-4^{-i}|r_j||\le 4\pi/\hat b_j$ for all $j,i,h$; 
with (\ref{eq4iadm}) we obtain 
$|R_{j,i,h}|\ge 4^{-i}|r_j| - 4\pi/\hat b_j
= 4^{-i}|r_j| (1- 4\pi 4^i/(|r_j|\hat b_j))
\ge \frac12 4^{-i} |r_j|$, and analogously 
$|R_{j,i,h}|\le \frac32 4^{-i}|r_j|$. Therefore
$|R_{j,i,h}|\sim 4^{-i}|r_j|$.

We localize the magnetic field in an appropriate subset of  each of the $R_{j,i,h}$. 
This is done by the family of inner rectangles  $(\hat R_{j,i,h})_{h=1,\dots, 4^i}$ which we now define:
we let  $\hat R_{j,i,h}$ have the same center and aspect ratio as $R_{j,i,h}$,
and area given by
 \begin{equation*}
 |\hat R_{j,i,h}| \frac{\kappa}{\sqrt2} = |R_{j,i,h}| \hat b_j\,.
 \end{equation*} 
We recall that $\hat b_j$ is defined so that $\hat b_j|r_j|=\kappa |\hat r_j|/\sqrt2\in 2\pi\N$,
in particular $\hat b_j\le \kappa/\sqrt2$.
Therefore, $\hat R_{j,i,h}\subset R_{j,i,h}$, and the
locally optimal field $\kappa/\sqrt2$ carries 
over $\hat R_{j,i,h}$ the same flux that $\hat b_j$ carries over $R_{j,i,h}$.
At the level $i=0$ this definition gives $\hat R_{0,j}=\hat r_j$.

To estimate the size of the rectangles we define
\begin{equation*}
 d_i:=2^{-i}d_0\hskip5mm  \text{ and }\hskip5mm
  \rho_i:=\rho_0 2^{-i} \,.
\end{equation*}
Then $|R_{j,i,h}|\sim 4^{-i}|r_j|\sim d_i^2$ and, correspondingly,   $|\hat R_{j,i,h}|
=|R_{j,i,h}|\,|\hat r_j|/|r_j|
\sim \rho_i^2$ for all $i=1,\dots,I$.

\begin{figure}
 \begin{center}
  \includegraphics[height=5.5cm]{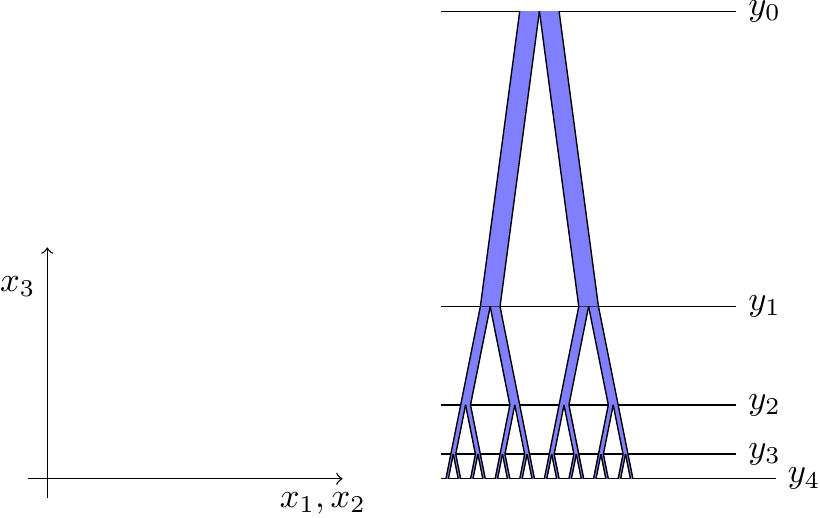}
\end{center}
\caption{Sketch of the vertical structure in Lemma 
\ref{lemmaconstrinternal}. 
The thickness of the tubes at stage $i$ is $\rho_i$, their horizontal separation $d_i$, the distance
between two vertical steps $t_i$.
}\label{figvertbr}
\end{figure}

  We now define the vertical structure, see Figure \ref{figvertbr}. The refinement steps, labeled by  $i$  in the decomposition
  of the rectangles, will describe the structure at different levels, which are labeled $y_i$. Precisely, for 
some $\theta\in (0,1)$ chosen below, we define
\begin{equation*}
 y_i:= T \theta^i, \hskip5mm
 t_i:=y_i-y_{i+1}= T \theta^i (1-\theta)\,.
\end{equation*}
By (\ref{eqconstrIcond}), if $\theta\ge 1/4$ we obtain $t_i\gsim d_i$ for all $i\le I$.

The explicit construction in $Q_L\times (y_{i+1}, y_i)$,
for $i=0, \dots, I-1$, is done using Lemma \ref{lemmatransformrectintorect}.
At each step we interpolate between $\hat R_{j,i,h}$ and the four corresponding 
$\hat R_{j,i+1,h_l}$, $l=1,2,3,4$.
Since the side length of $R_{j,i,h}$ is controlled by $d_i$, 
 the side length of the inner rectangle where the field is located is controlled by $\rho_i$,
 and $d_i\lsim t_i$, 
the  energy in $Q_L\times (y_{i+1}, y_i)$ is bounded by
\begin{equation*}
F^\Int[\chi,B,Q_L\times (y_{i+1}, y_i)]\lsim
 \sum_{j=1}^{N^2} \sum_{h=1}^{4^i}  \left[ \kappa \rho_it_i +\kappa^2 \rho_i^2\frac{d_i^2}{t_i}
\right] \lsim N^2\kappa \rho_0t_0(2\theta)^i  + N^2\kappa^2 \frac{\rho_0^2 d_0^2}{t_0 (4\theta)^i} \,.
\end{equation*}
The series converges for all $\theta\in (1/4, 1/2)$. Choosing $\theta=1/3$
we conclude that the  energy in $Q_L\times (y_I,T)$ is bounded by
\begin{equation*}
F^\Int[\chi,B,Q_L\times (y_{I}, T)]\lsim
  \sum_{i=1}^I   \sum_{j=1}^{N^2} \sum_{h=1}^{4^i}\left[ \kappa \rho_it_i +\kappa^2 \rho_i^2\frac{d_i^2}{t_i}
\right] \lsim \kappa \rho_0 TN^2+\kappa^2\rho_0^2 \frac{d_0^2 N^2}{T}\,.
\end{equation*}
We still have to consider the region $Q_L\times(0, y_I)$. Here we set $B'=0$ and $(\chi,B)(x',x_3)=(\chi,B)(x', y_I)$.
The only energy contribution comes from the surface term, and is $N^24^I\kappa\rho_I y_I$. 
Since $y_I\lsim t_I$, this is controlled by the last summand in the series,
and therefore does not change the scaling of the total energy.

 We finally estimate the boundary term.
 Using Lemma \ref{lemmaestimatehminus12} we obtain
 \begin{equation*}
  \|\sum_{j=1}^{N^2}\sum_{h=1}^{4^I}\left( \frac{\kappa}{\sqrt2} \uno_{\hat R_{j,I,h}} - \hat b_j \uno_{R_{j,I,h}} \right)
  \|_{H^{-1/2}(Q_L)}^2 \lsim \kappa^2 \sum_{j=1}^{N^2}\sum_{h=1}^{4^I} |\hat R_{j,I,h}|^{3/2}
  \lsim  \frac{N^2\kappa^{2}\rho_0^3}{2^I}  \,.
 \end{equation*}
 We observe that $\sum_h \uno_{R_{j,I,h}}=\uno_{r_j}$, by the construction of the rectangles. 
Inserting, from the definition (\ref{eqconstrIcond}),
$2^{-I}\sim (\kappa\rho_0)^{-1}+d_0/T$,
concludes the proof of (\ref{eqhm12QLd}).

To prove (\ref{eqestchiuppernbd2}) from Lemma \ref{lemmatransformrectintorect} 
it suffices to show that $1/\kappa\lsim \rho_I$, i.e., $\kappa\rho_I\gsim 1$.
Since $\kappa\rho_I=2^{-I}\kappa\rho_0 $, this follows immediately from (\ref{eqconstrIcond}).
\end{proof}

\begin{proof}[Proof of Theorem \ref{theochiupper1}]
We fix $k\in\N$, set $N:=2^{k}$, and let $\{R_{k,j}\}_{1\le j\le N^2}$ be the 
rectangles given in Lemma \ref{lemmasubdivideflux} applied to $R_{0,1}:=Q_L=(0,L)^2$ with $B_\ast:=\bext$.
We assume
\begin{equation}\label{eqcondNcostr}
 N^2\le \frac{\bext L^2}{8\pi}
\end{equation}
so that $|R_{k,j}|\sim L^2/N^2$.
We first localize the flux to the central part of each of the rectangles. This is the step 
in which the two regimes differ. Fix 
a factor $\gamma\in (0,1]$, chosen below,
and let $r_j$ be a rectangle with the same center and aspect ratio as $R_{k,j}$, scaled by a factor $\gamma$.
We shall concentrate the 
entire flux over $R_{k,j}$ into the smaller rectangle $r_j$, so that the magnetic field over $r_j$ is $\hat b:=\bext /\gamma^2$;
from $\bext |R_{k,j}|\in 2\pi \Z$ we obtain $\hat b |r_j|\in2\pi\Z$.
Since $\bext=\sum_{j=1}^{N^2} \bext \uno_{R_{k,j}}$, 
 Lemma \ref{lemmaestimatehminus12} yields
\begin{alignat}1
   \|\sum_{j=1}^{N^2} \hat b \uno_{r_j} - \bext  \|_{H^{-1/2}(Q_L)}^2 &
  \lsim  \sum_j  \hat b^2|r_j|^{3/2} \uno_{\gamma<1}
  \lsim \bext^2 \frac{L^3}{N \gamma} \uno_{\gamma<1}\,.\label{eqhm12QLc}
 \end{alignat}
 The factor $ \uno_{\gamma<1}$ represents the fact that this term is only present if $\gamma<1$. In the case $\gamma=1$,
 indeed, we have $\sum \uno_{r_j}=1$ on $Q_L$, and therefore this term vanishes.
 
 We then use Lemma \ref{lemmaconstrinternal}
 with the given set of $r_j$, $d_0=\gamma\frac{L}{N}$, $\kappa$, $L$, $N$ as above, and thickness $T/2$, 
on the set $Q_L\times(0,T/2)$ (the other half is symmetric and not discussed explicitly).
The inner rectangles  are chosen so that $\hat r_j$ has the same center
and aspect ratio as $r_j$ and area given by $\kappa|\hat r_j|/\sqrt2=\hat b |r_j|$, correspondingly
the length scale is
\begin{equation*}
\rho_0:=d_0  \left(\frac{\hat b \sqrt2}{\kappa}\right)^{1/2}
=\frac LN \left(\frac{\bext\sqrt2}{\kappa}\right)^{1/2}\,.
\end{equation*}
Since $\hat r_j$ needs to be a subset of $r_j$, this is possible only if 
$\hat b\le \kappa/\sqrt2$, therefore it is only possible if $\gamma$ is chosen such that
 \begin{equation}\label{eqconstrgammacond}
  \left(\frac{\bext\sqrt2}{\kappa}\right)^{1/2}\le\gamma \le 1\,.
 \end{equation}
Since $\bext\le \kappa/2$, this set is non empty.
At the same time, the condition (\ref{eqkapparho0}) is satisfied provided that
\begin{equation}\label{eqkapparho0subst}
 (\bext\kappa)^{1/2} \frac{L}{N} \ge1\,.
\end{equation}
 To estimate the boundary term we combine (\ref{eqhm12QLc}) with (\ref{eqhm12QLd})
 \begin{align*}
  \| B_3(\cdot,0)-\bext\|_{H^{-1/2}(Q_L)}^2 
 & \le 2 \|B_3(\cdot,0)-\sum_{j=1}^{N^2} \hat b \uno_{r_j}\|_{H^{-1/2}(Q_L)}^2 
  +2\|\sum_{j=1}^{N^2} \hat b \uno_{r_j}-\bext\|_{H^{-1/2}(Q_L)}^2 
 \end{align*}
and obtain
 \begin{equation}
  \|B_3(\cdot,0) - \bext
  \|_{H^{-1/2}(Q_L)}^2 \lsim \bext^2 \frac{L^3}{N \gamma}\uno_{\gamma<1}+  \kappa\rho_0^2N^2  + \frac{\kappa^{2}\rho_0^3N^2d_0}{T}   \,.
 \end{equation} 
 We then extend $B$ to $Q_L\times(-\infty,0)$ by Lemma \ref{lemmah12}, 
 and define $\chi$ and $B$ on $Q_L\times (T/2,\infty)$ by symmetry,
 $(\chi,B)(x',T/2+z)=(\chi,B)(x',T/2-z)$.
 Recalling (\ref{eqfintpropcostr}) we see that the total energy is  bounded by
 \begin{equation*}
  F[\chi,B]\lsim
  \kappa \rho_0 TN^2+\kappa^2 \frac{\rho_0^2 d_0^2N^2}{T}
  +\bext^2 \frac{L^3}{N \gamma}\uno_{\gamma<1}+  \kappa\rho_0^2N^2  + \frac{\kappa^{2}\rho_0^3N^2d_0}{T}\,.
 \end{equation*}
Since $\rho_0\le d_0$, the last term is smaller than the second one and therefore can be neglected.
 Dividing by the area $L^2$ and 
 inserting the definitions $\rho_0=L \bext^{1/2}/(N \kappa^{1/2})$ and
 $d_0= L\gamma/N$ gives 
\begin{equation}\label{equbfchibsdf}
 \frac{F[\chi,B]}{L^2}\lsim \kappa^{1/2} \bext^{1/2} \frac{TN}{L}+ \kappa \bext \frac{L^2\gamma^2}{N^2T}+
\bext^2 \frac{L}{N \gamma}\uno_{\gamma<1}+  
\bext 
\,.
\end{equation}
The construction is possible 
for all $N\ge 1$, $\gamma\in(0,1]$ which obey   (\ref{eqcondNcostr}), (\ref{eqconstrgammacond}), and (\ref{eqkapparho0subst}). To conclude the proof it suffices to choose these 
parameters appropriately.

\medskip 
\noindent
{\it Choice of the parameters: intermediate regime.}
We assume here 
\begin{equation}\label{eqregimeint}
\bext\ge \frac{\kappa^{5/7}}{\sqrt2T^{2/7}}\,.
\end{equation}
In this regime, we balance the first two terms in (\ref{equbfchibsdf}) by choosing the length scale $L/N$ as
\begin{equation*}
\gamma=1 \hskip1cm \text{ and }\hskip1cm N=\inf\{2^k: k\in\N, 2^k\ge N_*\} \hskip5mm \text{ where }\hskip5mm
\frac{N_*}{L}=\frac{(\kappa \bext)^{1/6}}{\alpha T^{2/3}}\,,
\end{equation*}
where $\alpha$ is a number of order 1 chosen below.
We assume $N_*\ge 1$ so that the rounding of $N_*$ to the next power of $2$ does not modify it by more than a factor of 2,
$N_*\le N\le 2N_*$. One obtains
\begin{align*}
 \frac{F[\chi,B]}{L^2}\lsim &(\kappa \bext)^{2/3} T^{1/3} 
 + \bext
 \lsim (\kappa \bext)^{2/3} T^{1/3} 
\end{align*}
where we used that one term disappears  because $\gamma=1$;
the $\bext$ term can be dropped since
$\bext / ((\kappa \bext)^{2/3} T^{1/3} )=(\bext/\kappa)^{1/3} (\kappa T)^{-1/3}\le 1$.

It remains to check that the choices made are admissible.
Condition  (\ref{eqkapparho0subst}) translates into
$\alpha (\kappa \bext T^2)^{1/3}\ge 2$. 
Using (\ref{eqregimeint}),
$\kappa T\ge 1$ and assuming $\alpha\ge4$ one can easily see that it is satisfied.
Since $N\le 2N_*$, condition (\ref{eqcondNcostr}) translates into
\begin{equation*}
\left(\frac{\bext^2 T^4}{\kappa}\right)^{1/3} \ge \frac{32\pi}{\alpha^2} \,.
\end{equation*}
Using first (\ref{eqregimeint}) and then $\kappa T\ge 1$  and $\kappa\le 1/2$ one can easily check that
the parenthesis is at least 4, 
therefore it suffices to choose $\alpha=8$.
Condition (\ref{eqconstrgammacond}) is immediate.
Finally, we check that $N_*\ge 1$. This is equivalent to
\begin{equation*}
 L\ge 8\frac{T^{2/3}}{(\kappa \bext)^{1/6}}\,.
\end{equation*}
\medskip 

\noindent
{\it Choice of the parameters: extreme regime.}
In this case we assume
\begin{equation}\label{eqregimeext}
\bext\le \frac{\kappa^{5/7}}{\sqrt2T^{2/7}}\,.
\end{equation}
In this regime, we can actually make the three first terms in (\ref{equbfchibsdf}) balance 
by choosing the length scale $L/N$and $\gamma$ according to
\begin{equation*}
\gamma=2^{1/4}\frac{T^{1/7}\bext^{1/2}}{\kappa^{5/14}}
\hskip5mm \text{ and }\hskip5mm   N=\inf\{2^k: k\in\N, 2^k\ge N_*\} \hskip5mm \text{ where }\hskip5mm
\frac{N_*}{L}=\frac{\bext^{1/2}}{\alpha\kappa^{1/14}T^{4/7}}\,,
\end{equation*}
for some $\alpha>0$ chosen below.
Again, we require $N_*\ge 1$ so that $N_*\le N\le 2N_*$.
This gives $d_0\sim T^{5/7}/\kappa^{2/7}$ and
\begin{equation*}
 \frac{F[\chi,B]}{L^2}\lsim  \bext\kappa^{3/7} T^{3/7} 
 + \bext\lsim \bext \kappa^{3/7}  T^{3/7} 
\end{equation*}
since $\kappa T\ge 1$.

We turn to checking that the choices made are admissible.
The assumption (\ref{eqregimeext})  is equivalent to $\gamma\le 1$.
The other inequality in (\ref{eqconstrgammacond}) is fulfilled by
 $\kappa T\ge 1$ (this is the reason for inserting the factors $2^{1/4}$ and $\sqrt2$
 in the definition of $\gamma$ and (\ref{eqregimeext})).
Condition (\ref{eqkapparho0subst}) becomes
$\alpha (\kappa T)^{4/7}\ge 2$, which is  true for any $\alpha\ge 2$.
Condition (\ref{eqcondNcostr})
becomes 
\begin{equation*}
 \frac{\alpha^2}{32\pi} \frac{(\kappa T)^{8/7}}{\kappa}\ge 1\,,
\end{equation*}
which again is satisfied if $\alpha=8$.
 The fact that $N_*\ge 1$ translates into
 \begin{equation*}
 L\ge8 \frac{T^{4/7}\kappa^{1/14}}{ \bext^{1/2}}\,.
\end{equation*}
This concludes the proof.
\end{proof}

\subsection{Construction for the Ginzburg-Landau functional}
We finally give the upper bound construction for the Ginzburg-Landau functional. 
We start from the constructions given in Section \ref{secconstrsharp} for
the sharp-interface functional $F$.

The first step is to construct the vector potential $A$ from the magnetic field $B$. We use the following lemma,
which is a variant of Hodge's decomposition in the current geometry.
\begin{lemma}\label{lemmaconstructA}
  Let $B\in L^2_\loc(\R^3;\R^3)$, $Q_L$ periodic, such that $\Div B=0$
  distributionally and $\int_{Q_L\times \R} |B|^2 dx<\infty$.
Then there is $A\in W^{1,2}_\loc(\R^3;\R^3)$, also $Q_L$-periodic, such that
\begin{equation}
  \nabla \times A = B \,.
\end{equation}
\end{lemma}
\begin{proof}
From Lemma \ref{lemmaHmeno12} we obtain
 \begin{equation*}
   \int_{Q(z)}B_3\,dx'=0\text{ for all }z\in \R
 \end{equation*}
 (in the sense of traces).
  We define
  \begin{equation*}
    \hat B(x) := B(x) - \frac{1}{L^2} \int_{Q(x_3)} B \,dx'
  \end{equation*}
  so that all components of $\hat B$ have average zero on each cross
  section $Q(z)$. Notice that $\hat B_3=B_3$, and $\nabla'\hat
  B=\nabla'B$. This implies
   $\hat B\in 
  L^2$, and $\Div \hat B=0$. 

  The rest of the proof gives a construction of $\hat A$ such that $\hat B=\nabla\times \hat A$.
  If one could solve directly $\Delta \Psi= \hat{B}$ with
$\Div \Psi=0$, then  $\nabla\times\Psi$ would work. Since we are working in an unbounded 
domain with mixed boundary conditions, 
for completeness we give an explicit construction of $\hat A$ based on Fourier series.

We Fourier transform to obtain coefficients $\hat b(k)$ such that
  \begin{equation*}
    \hat B(x)=\int_\R dk_3\sum_{k'\in 2\pi\Z^2/L} e^{ik\cdot x} \hat b(k)\,.
  \end{equation*}
  The transformation we just performed ensures that $\hat b(k)=0$ whenever
  $k'=0$ (this is the reason to consider $\hat B$ instead of $B$). 
Further, $k\cdot \hat b(k)=0$ for all $k$.

  We define
  \begin{equation*}
    \hat a(k):=
    \begin{cases}\displaystyle
\frac{-ik\times \hat b(k)}{k^2} & \text{ if } k'\ne 0\\
\displaystyle
0 & \text{ if }k'=0\,,
    \end{cases}
  \end{equation*}
so that for all $k$
\begin{equation*}
  ik\times \hat a(k) = 
    \hat b(k)\,.
\end{equation*}
The family $\hat a(k)$ also corresponds to a converging Fourier series,
since $|\hat a(k)|\le L|\hat b(k)|/(2\pi)$ (this is the step where it is important that
we requested $k'\ne 0$, and not merely $k\ne 0$). 

In real space, we set
\begin{equation*}
  \hat A(x):=\int_\R dk_3\sum_{k'\in 2\pi\Z^2/L} e^{ik\cdot x} \hat
  a(k)\,. 
\end{equation*}
Clearly $\hat A\in W^{1,2}_\loc(\R^3;\R^3)$, it is $Q_L$-periodic and obeys
\begin{equation*}
  \nabla\times \hat A = \hat B\,,
\end{equation*}
both sides being in $L^2(Q_L\times\R;\R^3)$. We finally set
\begin{equation*}
  A(x)=\hat A(x) - \frac{1}{L^2}\int_{Q_L\times(0,x_3)} (e_3\times B)dx\,.
\end{equation*}
The correction depends only on $x_3$. Therefore, recalling $\int_{Q(x_3)} B_3 dx'=0$,
\begin{eqnarray*}
  \nabla\times A &=& \nabla\times \hat A - \frac{1}{L^2}\int_{Q(x_3)} e_3\times
  (e_3\times B)dx'\nonumber\\
 &=& \hat B  + \frac{1}{L^2}\int_{Q(x_3)} B\, dx'= B\,.
\end{eqnarray*}
This concludes the proof.
\end{proof}

\begin{theorem}\label{theoupperbGL}
For any $\bext,  \kappa,L,T>0$ such that
\begin{equation*}
4\bext\le \kappa\le\frac12\hskip1cm\text{ and }\hskip1cm
  \kappa T\ge 1\,,
\end{equation*}
$L$ sufficiently large (in the sense of (\ref{eqLadmissible}) and  $ \bext L^2\in 2\pi \Z$
there is a pair $(u,A)\in H^1_{\mathrm{per}}$, such that
\begin{equation*}
 E[u,A]\lsim \min \left\{\bext \kappa^{3/7} T^{3/7} L^2, 
\bext^{2/3} \kappa^{2/3}  T^{1/3}L^2\right\} \,.
\end{equation*}
\end{theorem}
\begin{proof}
 Let $\chi$ and $B$ be the functions constructed in Theorem \ref{theochiupper1}.
Starting from $B$, we obtain $A$ from
Lemma \ref{lemmaconstructA}. Consider now the superconducting
domain $\omega:=\{x\in Q_{L,T}: \chi(x)=0\}$. Here $A$ is a curl-free vector field, hence it is locally
the gradient of some potential $\theta$. The domain is multiply
connected, but the flux of $B$ across each tube is an integer multiple
of $2\pi$, hence we can globally write $A$ as the gradient of a
multi-valued function $\theta$, such that $\theta$ mod $2\pi$ is
single-valued. 

We set $\rho=0$ in the normal phase, and let it grow to 1 in the
superconducting phase, on a length scale $1/\kappa$,
\begin{equation*}
  \rho(x) := \min\{1, \kappa^2\,\dist^2(x,\omega)\}. 
\end{equation*}
The distance function is understood, as usual, $Q_L$-periodic in the first two components. Finally, we set
\begin{equation*}
  u(x):=\rho^{1/2}(x) e^{i\theta(x)}\,.
\end{equation*}

Since $\nabla \theta=A$ whenever $\rho\ne 0$, we have 
\begin{equation*}
  |\nabla_Au|^2 = |\nabla\rho^{1/2}|^2. 
\end{equation*}
The $B'$  part of the energy is identical, and so is  the outer field. It
remains to treat the coupling term. 
In $\omega$ we have $\rho=0$ and  $\chi=1$, hence
\begin{equation*}
  \left(B_3-\frac{\kappa}{\sqrt2} (1-\rho)\right)^2 =
  \chi\left(B_3-\frac{\kappa}{\sqrt2} \right)^2 \,.
\end{equation*}
Outside $\omega$ we have $B=0$, hence
\begin{equation*}
  \left(B_3-\frac{\kappa}{\sqrt2} (1-\rho)\right)^2 =
  \frac{\kappa^2}{2}(1-\rho)^2
\end{equation*}
The first term is exactly the one appearing in $F$. 
Therefore the $E[u,A]\le F[\chi,B]+E_S$, where 
\begin{equation*}
 E_S:= \int_{Q_{L,T}} \left[|\nabla\rho^{1/2}|^2 + \kappa^2 (\chi-(1-\rho))^2\right]dx\,.
\end{equation*}
Let  $(\omega)_{1/\kappa}$ be a $1/\kappa$-neighbourhood of $\omega$.
Then in $\omega$ we have $\chi=1=1-\rho$, outside $(\omega)_{1/\kappa}$ we
have $\chi=0=1-\rho$, and recalling $|\nabla \rho^{1/2}|\le \kappa$ we obtain
\begin{equation*}
 E_S \le2\kappa^2 |(\omega)_{1/\kappa}\setminus\omega|\,.
\end{equation*}
Recalling (\ref{eqestchiuppernbd}) we conclude $E_S\lsim \kappa \int_{Q_{L,T}} |D\chi|\le F[\chi,B]$.
This concludes the proof.
\end{proof}

\bibliographystyle{alpha-noname}
\bibliography{COS}

\end{document}